\documentclass{amsart}
\pdfoutput=1
\usepackage{amsmath,amsthm,amsfonts,latexsym,epsfig,longtable,fullpage}

\usepackage{graphicx}

\usepackage{enumerate}

\newtheorem{theorem}{Theorem}[section]
\newtheorem{lemma}[theorem]{Lemma}
\newtheorem{proposition}[theorem]{Proposition}
\newtheorem{corollary}[theorem]{Corollary}

\theoremstyle{definition}

\theoremstyle{remark}

\numberwithin{equation}{section}

\newcommand{\Hu}{\mathbf{H}}
\newcommand{\Hy}{\mathbf{G}}
\newcommand{\Aut}{\mathrm{Aut}\,}

\newcommand{\cyc}{\mathsf{cyc}}

\newcommand{\Pbar}{\mathsf{Pic}}
\newcommand{\ld}{\ldots}
\newcommand{\cd}{\cdots}
\newcommand{\pr}{\prime}
\newcommand{\pp}{p^{\perp}}
\newcommand{\Pp}{P^{\perp}}

\newcommand{\bfp}{{\bf p}}
\newcommand{\bfpp}{{\bf p^{\perp}}}

\newcommand{\bbZ}{\mathbb{Z}}

\newcommand{\cM}{\mathcal{M}}
\newcommand{\cA}{\mathcal{A}}
\newcommand{\cH}{\mathcal{H}}
\newcommand{\cF}{\mathcal{F}}
\newcommand{\cU}{\mathcal{U}}
\newcommand{\cD}{\mathcal{D}}

\newcommand{\cS}{\mathcal{S}}
\newcommand{\cT}{\mathcal{T}}
\newcommand{\cC}{\mathcal{C}}
\newcommand{\cB}{\mathcal{B}}

\newcommand{\cP}{\mathcal{P}}

\newcommand{\bx}{\mathbf{x}}

\newcommand{\bp}{\mathbf{p}}
\newcommand{\bq}{\mathbf{q}}
\newcommand{\bc}{\mathbf{c}}

\newcommand{\al}{\alpha}
\newcommand{\be}{\beta}
\newcommand{\ga}{\gamma}
\newcommand{\ka}{\kappa}
\newcommand{\Ga}{\Gamma}
\newcommand{\la}{\lambda}
\newcommand{\La}{\Lambda}
\newcommand{\Up}{\Upsilon}

\newcommand{\si}{\sigma}
\newcommand{\de}{\delta}
\newcommand{\De}{\Delta}
\newcommand{\Th}{\Theta}
\newcommand{\Om}{\Omega}
\newcommand{\sipr}{\si^{\pr}}
\newcommand{\whPsi}{\widehat{\Psi}}

\begin{document}

\title[Contents of partitions and permutation factorizations]{Contents of partitions and the combinatorics of permutation factorizations in genus $0$} 

\date{\today}


\author{S. R. Carrell}
\address{Dept. of Combinatorics and Optimization, University of Waterloo, Canada}
\curraddr{}
\email{srcarrel@uwaterloo.ca, ipgoulde@uwaterloo.ca}

\author{I. P. Goulden}
\thanks{The work of IPG was supported by an NSERC Discovery Grant.}



\begin{abstract} 
The central object of study is a formal power series that we call the {\em content series}, a symmetric function involving an arbitrary underlying formal power series $f$ in the contents of the cells in a partition. In previous work we have shown that the content series satisfies the KP equations. The main result of this paper is a new partial differential equation for which the content series is the unique solution, subject to a simple initial condition. This equation is expressed in terms of families of operators that we call $\cU$ and $\cD$ operators, whose action on the Schur symmetric function $s_{\la}$ can be simply expressed in terms of powers of the contents of the cells in $\la$. Among our results, we construct the $\cU$ and $\cD$ operators explicitly as partial differential operators in the underlying power sum symmetric functions. We also give a combinatorial interpretation for the content series in terms of the Jucys-Murphy elements in the group algebra of the symmetric group. This leads to an interpretation for the content series as a generating series for branched covers of the sphere by a Riemann surface of arbitrary genus $g$. As particular cases, by suitable choice of the underlying series $f$, the content series specializes to the generating series for three known classes of branched covers: Hurwitz numbers, monotone Hurwitz numbers, and $m$-hypermap numbers of Bousquet-M\'elou and Schaeffer. We apply our pde to give new and uniform proofs of the explicit formulas for these three classes of numbers in genus $0$.
\end{abstract}

\maketitle

\section{Introduction and the content series}\label{sec1}

A $k$-tuple of positive integers $\al = (\al_1,\ld ,\al_k)$ with $\al_1 \ge \cd \ge \al_k \ge 1$ and $\al_1 + \ld + \al_k = n$ is called a (integer) {\em partition} of $n$ with $k$ {\em parts}. We use the notation $\al\vdash n$, $|\al|=n$ and $l(\al )=k$ throughout the paper. If $\al$ has $m_i$ parts equal to $i$ for each $i\ge 1$, then we may also write $\al =1^{m_1}2^{m_2}\cd$. The set of all partitions, including a single partition of $0$ with $0$ parts, is denoted by $\cP$.

We will use various results for symmetric functions in the countable set of variables $\bx = (x_1,x_2,\ld )$, and refer the reader to~\cite{mac} for proofs and further details.  The $i$th {\em power sum} symmetric function is given by $p_i=\sum_{j\ge 1} x_j^i$, $i\ge 1$, and $p_0=1$. We write $\bp = (p_1,p_2,...)$, and $p_{\al} =p_{\al_1}p_{\al_2}\cd$, where $\al$ is a partition with parts $\al_1,\al_2, \ld$ (and in general write $g_{\al}=g_{\al_1}g_{\al_2}\cd$ for any $g_i$, $i\ge 1$). The {\em Schur functions} $s_{\la}$ for partitions $\la\in\cP$ are related to the $p_{\al}$ by the inverse linear relations
\begin{equation}\label{Schurpower}
s_{\la}=\sum_{\al \vdash n} \frac{| \cC_{\al} |}{n!} \chi^{\la}_{\al} p_{\al},\quad  \la \vdash n,\qquad \qquad \qquad p_{\al}=\sum_{\la \vdash n} \chi^{\la}_{\al} s_{\la},\quad \al \vdash n,
\end{equation}
where $\cC_{\al}$ is the conjugacy class in $\cS_n$ consisting of all permutations with disjoint cycle lengths specified by the parts of $\al$, and $\chi^{\la}_{\al}$ is the value of the irreducible character $\chi^{\la}$ of $\cS_n$ indexed by $\la$, evaluated at any element of the conjugacy class $\cC_{\al}$. We will use the nonstandard notation $s_{\la}(\bp)$ to denote the expression for $s_{\la}$ in terms of $\bp$ given in (\ref{Schurpower}) above (instead of writing $s_{\la}(\bx)$ to denote the symmetric function $s_{\la}$ in the underlying variables $\bx$ in the more standard way). 
The Hall inner product for symmetric functions is defined for the basis of Schur functions by $\langle s_{\la},s_{\mu} \rangle = \de_{\la ,\mu}, \; \la,\mu \in \cP$. For any left linear operator $\cA$ on symmetric functions, we define $\cA^{\perp}$ to be the adjoint operator, so we have
\begin{equation}\label{defnadj}
\langle \cA f,g \rangle = \langle f, \cA^{\perp} g \rangle
\end{equation}
for any symmetric functions $f,g$. It is straightforward to check that $\pp_i=i \frac{\partial}{\partial p_i}$, $i\ge 1$ and we let $\bfpp =(\pp_1,\pp_2,\ld )$. 

Throughout this paper we will consider the combinatorial quantity called {\em content}, that is associated with the diagram of a partition. A cell $\Box$ in the diagram of a partition that occurs in row $i$ (indexed from the top) and column $j$ (indexed from the left) has content $j-i$, written $c(\Box)=j-i$. For example, the left hand side of Figure~\ref{Figure.contentDef} is a diagram of the partition $(5,3,3,1)$ with the value of the content in each of its cells. We use the notation $\Box \in \la$, say in the range of a product or summation, to mean that $\Box$ ranges over all cells in the diagram of the partition $\la$. We will also encounter {\em skew} partitions $\la / \mu$, where $\la$ and $\mu$ are
partitions such that $\la_i \geq \mu_i$ for all $i \geq 1$. The cells in $\la / \mu$ are those which are contained in $\la$ but not in $\mu$; however, we shall view the cells in the
skew partition as being drawn on the integer lattice so that the content of any cell in the skew partition $\la / \mu$ is the same as its content in $\la$. For example, the right hand side of Figure~\ref{Figure.contentDef} is a diagram of the skew partition $(5,3,3,1)/(5,2,1)$ along with the content of each of its cells. We define $\bc(\la / \mu)$ to be the multiset of contents of cells in $\la / \mu$.

\begin{figure}[h]
	\centering
	\includegraphics[width=0.7\textwidth]{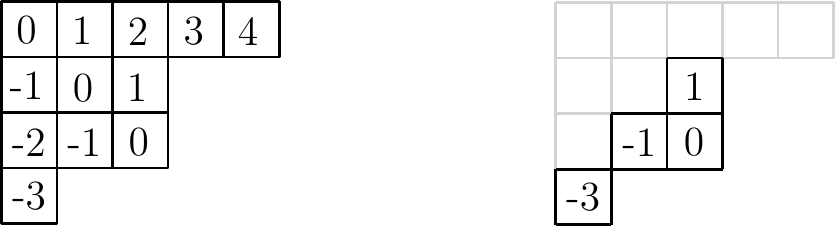}
	\caption{Content of cells in a partition and skew partition}
	\label{Figure.contentDef}
\end{figure}

The main object of study in this paper is the formal power series
\begin{equation}\label{defPhifseries}
\Phi^{f(x)}(y,z,\bp)  = \sum_{n\ge 0} \frac{z^n}{n!} \sum_{\la \vdash n} \chi^{\la}_{(1^n)} s_{\la} (\bp) \prod_{ \Box \in \la} f(yc(\Box)),
\end{equation}
in variables $y, z, \bp$, where $f(x)$ is an arbitrary formal power series in a single variable $x$ (used as a place holder only, to avoid confusion with $y,z$, etc.). We will refer generically to $\Phi^{f(x)}$ as the {\em content series} throughout this paper. The related series
\begin{equation*}
\Up^{f(x)}(y,\bp,\bq)  = \sum_{n\ge 0} \sum_{\la \vdash n} s_{\la}(\bq) s_{\la} (\bp) \prod_{ \Box \in \la} f(yc(\Box)),
\end{equation*}
in variables $y,\bp,\bq =(q_1,q_2,\ld)$ has appeared previously. In particular, $\Up$ was shown in~\cite{c1} to be a $\tau$-function for the $2$-Toda hierarchy, which in turn means that it satisfies a countable set of pdes in which partial differentiation is applied both in the elements of $\bp$ and $\bq$. The series $\Up$ is also a $\tau$-function for the KP hierarchy (see, e.g., \cite{gj2}), in which partial differentiation is applied in one of $\bp$ or $\bq$ only.

The content series $\Phi^{f(x)}$ can be obtained from $\Up^{f(x)}$ by the substitution $q_1=z$, $q_i=0$, $i\ge2$, since
\[ s_{\la}(z,0,\ld)= \frac{z^n \chi^{\la}_{(1^n)}}{n!},\qquad \la \vdash n .\]
The fact that $\Phi^{f(x)}$ is obtained from $\Up^{f(x)}$ by substitution for $\bq$ implies that $\Phi^{(f(x)}(y,z,\bp)$ is a $\tau$-function for the KP hierarchy, though it is not a $\tau$-function for the $2$-Toda hierarchy. 

A key combinatorial element of the paper is provided by the symmetric group $\cS_n$ acting on the set $[n]=\{ 1,\ld ,n\}$, $n\ge 0$. We use the notation $\cC_{\al}$ both for the conjugacy class in $\cS_n$, and also for the formal sum over that conjugacy class in the group algebra of $\cS_n$, where the particular meaning of each usage will be clear from the context. In the group algebra of the symmetric group $\cS_{n}$, let $J_i$ denote the {\em Jucys-Murphy} element 
\begin{equation}\label{JMdefn}
J_i= (1\; i) + (2\; i) + \ld + (i-1 \; i), \quad i=1,\ld ,n.
\end{equation}

 On the other hand, a key algebraic element of the paper for the various formal power series in $\bp$ and other indeterminates is to focus on monomials in the elements of $\bp$. We introduce some terminology to help with this. For a partition $\al$, it clear that the monomial $p_{\al}$ has (total) {\em degree} $l(\al)$ in the elements of $\bp$. We also say that $p_{\al}$ has {\em weight} $| \al |$ (note that for the content series, this is also the degree of the indeterminate $z$). We will consider partial differential operators in the elements of $\bp$ of the form $p_{\ga} \pp_{\al}$, and note that, in the above terminology, this operator has the following effect on a monomial in the elements of $\bp$: it changes the weight by $| \ga | - | \al |$ (it also changes the degree by $l(\ga)-l(\al)$, but that will not be as important).

In Section~\ref{sec2} of the paper we introduce some families of operators on symmetric functions that are described by their action on the Schur functions, including two related families of particular importance, the $\cU$ and $\cD$ operators. We then give our main result, an equation for the content series in terms of these operators. The equation in the main result is actually a (partial) differential equation due to the fact that the operators can be realized as a partial differential operators in the elements of $\bp$. In Section~\ref{s3} we give a combinatorial interpretation for the content series in terms of products of permutations via Jucys-Murphy elements, or equivalently an algebraic geometry interpretation in terms of branched covers of the sphere by a Riemann surface of arbitrary genus. This geometric point of view is then applied in Section~\ref{s4} to show that the content series specializes to the generating series for three previously studied problems concerning branched covers: Hurwitz numbers, monotone Hurwitz numbers, and $m$-hypermap numbers. In Section~\ref{UDsec}  we give explicit expressions for the $\cU$ and $\cD$ operators as differential operators in the elements of $\bp$, in which the $\cU$ operators increase weight by $1$, and the $\cD$ operators decrease weight by $1$. These explicit expressions are graded by total degree, and in Section~\ref{s6} we prove that the grading is actually indexed by genus from a Jucys-Murphy point of view. This allows us to give a completely combinatorial proof for the genus $0$ graded portion of the $\cU$ and $\cD$ operators. In Section~\ref{s8} we apply our genus interpretation for the operators to the main result, and thus obtain a partial differential equation for the genus $0$ portion of the content series. Also in Section~\ref{s8} we give some technical results for symmetric functions in a finite set of variables that will help us with the applications of our genus $0$ pde.

There are three applications of the genus $0$ pde, which are given in Section~\ref{s9}. These applications give us uniform {\em algebraic} proofs for the number of branched covers in genus $0$ for each of Hurwitz numbers, monotone Hurwitz numbers and $m$-hypermap numbers - i.e., for the three specializations of the content series that were described in Section~\ref{s4}. Our proof is new for the cases of Hurwitz numbers and the $m$-hypermap numbers of Bousquet-M\'elou and Schaeffer~\cite{bms}.  Motivated by this, Section~\ref{s9} begins with a discussion of {\em combinatorial} aspects of our methods and the impact of the underlying multiplication by Jucys-Murphy elements.

\section{A partial differential equation for the content series}\label{sec2}

Suppose that we have families of operators $\{ \cU_k \}_{k \ge 0}$ and $\{ \cD_k \}_{k \ge 0}$ whose action on Schur functions is given by
\begin{equation}\label{defUkDk}
\cU_k\, s_{\la} = \sum_{\substack{\Box:\\ \mu = \la + \Box}} c(\Box )^k s_{\mu},\qquad\qquad \cD_k\, s_{\la} = \sum_{\substack{\Box:\\ \la = \mu + \Box}} c(\Box )^k s_{\mu}, \qquad \la \vdash n, \quad n,k\ge 0.
\end{equation}
In the range of summations above we have used the notation $\la = \mu + \Box$ to mean that both $\la$ and $\mu$ are partitions and that the diagrams differ by a single cell $\Box$ (which must be the rightmost cell in some row of $\la$ and must be the bottommost cell in some column of $\la$). Note that the Murnaghan-Nakayama formula gives
\[ p_1 s_{\la} = \sum_{\substack{\Box:\\ \mu = \la + \Box}} s_{\mu},\qquad\qquad p_1^{\perp} s_{\la} = \sum_{\substack{\Box:\\ \la = \mu + \Box}}  s_{\mu}, \qquad \la \vdash n, \quad n\ge 0, \]
and so we may deduce that for $k=0$ such operators exist and are given by $\cU_0 = p_1$ and $\cD_0=\pp_1$; moreover, we deduce that $\cU_0$ increases weight (and degree) by $1$, and $\cD_0$ decreases weight  (and degree) by $1$. In \cite{l1}, Lassalle gives a method of recursively constructing the $\cU$ and $\cD$ operators which depends on the underlying variables $\bx$. In Section~\ref{UDsec} of this paper, using a similar method, we construct these operators without needing to make the variables $x_i$ explicit. This allows us to express them as differential operators in the $p_i$'s, via explicit summations of monomials in the $p_i$'s and $\pp_i$'s. Note that $\cD^{\perp}_k=\cU_k$, and of course $\cU^{\perp}_k=\cD_k$ (see~(\ref{defnadj})) for $k\ge 0$, so results for one of these classes immediately yields comparable results for the other class (via $\perp$), and consequently we will concentrate most of our attention on the $\cU_k$'s.

As an aside, the $\cU$ and $\cD$ operators are closely related to the {\em join-cut} operator
\begin{equation}\label{Deltadef}
\De = \tfrac{1}{2} \sum_{i,j\ge 1} \left(  p_i p_j \pp_{i+j} + p_{i+j} \pp_i \pp_j  \right) ,
\end{equation}
which was used in \cite{gj1} to count the number of factorizations of a permutation into transpositions (see also~\cite{goul}). It is clear that $\De$ is self-adjoint, but it is also an eigenoperator for Schur functions, as specified by
	\[ \Delta s_\la = \sum_{\Box \in \la} c(\Box) s_\la, \qquad \la\in\cP. \]
and can be used to construct the $\cU_k$ and $\cD_k$ recursively as follows:  Start with $\cU_0 = p_1$ and $\cD_0 = \pp_1$. Then, for $k\ge 1$, let $\cU_k = [\Delta, \cU_{k-1}]$ and $\cD_k = [\cD_{k-1}, \Delta]$ where $[\cA, \cB] = \cA\cB - \cB\cA$ is the commutator of any two operators $\cA$ and $\cB$. It is straightforward to check that this gives the operation on Schur functions that is specified in~(\ref{defUkDk}), but we have not been able to use this recursive method for the $\cU_k$ and $\cD_k$ to determine them explicitly. Note also that the join-cut operator is weight preserving for monomials in the elements of~$\bp$.

Two related classes of operators that use a similar range of summation are the {\em Sekiguchi-Debiard operators} $\{ \cC_k \}_{k \ge 0}$ and $\{ \cT_k \}_{k \ge 0}$. These are also eigenoperators for Schur functions, as specified by 
\begin{equation}\label{defCkTk}
\cC_k\, s_{\mu} = \left( \frac{n}{\chi^{\mu}_{1^n}} \sum_{\substack{\Box:\\ \mu = \la + \Box}} \!\!\!c(\Box )^k \chi^{\la}_{1^{n-1}} \right) s_{\mu}, \qquad \cT_k\, s_{\mu} = \left(  \frac{1}{(n+1) \chi^{\mu}_{1^n}} \sum_{\substack{\Box:\\ \la = \mu + \Box}}\!\!\! c(\Box )^k \chi^{\la}_{1^{n+1}} \right) s_{\mu},
\end{equation}
for $\mu \vdash n$ and $n,k\ge 0$. The operators $\cC_k$ and $\cT_k$ arise in the representation theory of the symmetric group, where their eigenvalues comprise the moments of the cotransition and transition measures, respectively, on partitions. Lassalle~\cite{l2} gave explicit formulas for some low order moments of the cotransition and transition measures (in the more general setting of measures with Jack parameter), of which the most relevant for our purposes are
\begin{equation}\label{transcotrans}
\cC_0 \, s_{\mu} = |\mu | s_{\mu}, \qquad \cC_1 \, s_{\mu} = 2\sum_{\Box \in \mu} c(\Box) s_{\mu} = 2\De s_{\mu}, \qquad \cT_0 \, s_{\mu} = s_{\mu}, \qquad \qquad \mu\in\cP.
\end{equation}

In the following theorem, the main result of this paper, we prove that the operators $\cU_k$ and $\cC_k$ can be used to construct a partial differential equation for the content series. The equation is expressed in terms of these operators, so the fact that it is a partial differential equation relies on the fact that the $\cU_k$ and $\cC_k$ can be written explicitly as partial differential operators in the elements of $\bp$, which will be carried out in Section~\ref{UDsec}.
\begin{theorem}\label{Phif_gpde}
If $f(x)=\sum_{i \ge 0} f_i x^i$ and $g(x)=\sum_{i \ge 0} g_i x^i$ with $g_0 \neq 0$, then the content series $\Phi^{fg^{-1}(x)}=\Phi^{fg^{-1}(x)}(y,z,\bp)$ is the unique solution to the partial differential equation
 \begin{equation}\label{mainpde}
 \left( \sum_{i \ge 0} f_i y^i \,\cU_i\right) \Phi^{fg^{-1}(x)} = z^{-1} \left( \sum_{i\ge 0} g_i y^i \cC_i \right) \Phi^{fg^{-1}(x)}
 \end{equation}
 that satisfies the initial condition $\Phi^{fg^{-1}(x)}(y,0,\bp)=1$.
 \end{theorem}
 
\begin{proof}
By reordering summations we obtain
\begin{align*}
\left( \sum_{i \ge 0} f_i y^i \,\cU_i\right) \Phi^{fg^{-1}(x)} &= \sum_{n\ge 0} \frac{z^n}{n!} \sum_{\la \vdash n} \chi^{\la}_{(1^n)}  \left( \prod_{ \Box \in \la} \frac{f(yc(\Box))}{g(yc(\Box))} \right) \sum_{\substack{\Box:\\ \mu = \la + \Box}} \left( \sum_{i \ge 0} f_i y^i c(\Box )^i \right) s_{\mu} \\
   &= \sum_{n\ge 0} \frac{z^n}{n!} \sum_{\mu \vdash n+1} \!\! \chi^{\mu}_{1^{n+1}} s_{\mu} \left( \prod_{ \Box \in \mu} \frac{f(yc(\Box))}{g(yc(\Box))} \right) \frac{1}{\chi^{\mu}_{1^{n+1}}} \sum_{\substack{\Box:\\ \mu = \la + \Box}} \left( \sum_{i \ge 0} g_i y^i c(\Box)^i \right) \chi^{\la}_{(1^n)},
\end{align*}
and~(\ref{mainpde}) follows immediately from~(\ref{defCkTk}). The initial condition is straightforward.
\end{proof}

There is a companion result to~(\ref{mainpde}) using the operators $\cD_k$ and $\cT_k$, and whose proof is almost identical. This result states that the content series $\Phi^{fg^{-1}(x)}$ satisfies the partial differential equation
  \begin{equation}\label{dualmainpde}
 \left( \sum_{i \ge 0} g_i y^i \,\cD_i\right) \Phi^{fg^{-1}(x)} = z \left( \sum_{i\ge 0} f_i y^i \cT_i \right) \Phi^{fg^{-1}(x)},
 \end{equation}
but the initial conditions are quite involved, requiring us to fix $\Phi^{fg^{-1}(x)}(y,z,\bp)\left. \right|_{p_1=0}$, and we have not been able to apply this to uniquely identify particular content series.

The following special case of Theorem~\ref{Phif_gpde} is an immediate corollary, slightly restated to avoid explicit mention of the operators $\cC_k$.
\begin{corollary}\label{Phifpde}
If $f(x)=\sum_{i \ge 0} f_i x^i$, then the content series $\Phi^{f(x)}=\Phi^{f(x)}(y,z,\bp)$ is the unique solution to the partial differential equation
\begin{equation}\label{fpde}
\left( \sum_{i \ge 0} f_i y^i \,\cU_i\right) \Phi^{f(x)} = \frac{\partial}{\partial z} \Phi^{f(x)}
\end{equation}
that satisfies the initial condition $\Phi^{f(x)}(y,0,\bp)=1$.
\end{corollary}
 
\begin{proof}
In the special case $g(x)=1$, Theorem~\ref{Phif_gpde} gives
\[  \left( \sum_{i \ge 0} f_i y^i \,\cU_i\right) \Phi^{f(x)} = z^{-1} \cC_0 \Phi^{f(x)} , \]
and the result follows from~(\ref{transcotrans}).
 \end{proof}
 
The companion result to Corollary~\ref{Phifpde}, analogous to~(\ref{dualmainpde}), is that if $g(x)=\sum_{i \ge 0} g_i x^i$, then the content series $\Phi^{g^{-1}(x)}=\Phi^{g^{-1}(x)}(y,z,\bp)$ satisfies the partial differential equation
\begin{equation}\notag
\left( \sum_{i \ge 0} g_i y^i \,\cD_i\right) \Phi^{g^{-1}(x)} = z \Phi^{g^{-1}(x)}.
\end{equation}
Again, due to the complicated initial condition, we have not been able to apply this equation to uniquely identify particular content series.

\section{Permutation factorizations, branched covers and the content series} \label{s3}

In order to give a combinatorial interpretation of the content series, we now turn to the group algebra of the symmetric group. In addition to the basis of conjugacy classes $\{ \cC_{\al}: \al\vdash n\}$, the center of the group algebra of $\cS_n$ also has a basis of orthogonal idempotents $\{ \cF_{\la}: \la \vdash n\}$. These bases are related by the inverse linear relations
\begin{equation}\label{idempconj}
\cF_{\la}= \frac{\chi^{\la}_{(1^n)}}{n!} \sum_{\al \vdash n} \chi^{\la}_{\al} \cC_{\al},\quad  \la \vdash n,\qquad \qquad \qquad \cC_{\al}= | \cC_{\al} | \sum_{\la \vdash n} \frac{ \chi^{\la}_{\al} }{ \chi^{\la}_{(1^n)} } \cF_{\la},\quad \al \vdash n,
\end{equation}
which are parallel to the relations~(\ref{Schurpower}) between the bases of Schur functions and power sums for symmetric functions of degree $n$.

One of the key results (e.g., see~\cite{ju}~\cite{mu}) about the Jucys-Murphy elements defined in~(\ref{JMdefn}) is that any symmetric function of  $J_1,\ld ,J_n$ is contained in the center of the group algebra of $\cS_n$ (though it is clear that the Jucys-Murphy elements themselves are {\em not} contained in the center). Moreover, for any symmetric function $G(x_1,\ld ,x_n)$, we have the following result for the eigenvalues of the orthogonal idempotent $\cF_{\la}$ and the multiset of contents $\bc(\la)$:
\begin{equation}\label{JMorthcontent}
G(J_1,\ld ,J_n) \cF_{\la} = G(\bc(\la)) \cF_{\la},\qquad \la\vdash n.
\end{equation}
 
The next result gives an explicit interpretation for the content series $\Phi^{f(x)}$ as a generating series that counts factorizations in the symmetric group. We use the notation $[A]B$ to denote the {\em coefficient} of $A$ in the expansion of $B$, in this case in the context of the group algebra of $\cS_n$.

\begin{proposition}\label{Phiffactninterp}
\[ \Phi^{f(x)}(y,z,\bp)  = \sum_{n\ge 0} \frac{z^n}{n!} \sum_{\al \vdash n} p_{\al} | \cC_{\al} | [  \cC_{\al} ] \prod_{i=1}^n f(yJ_i), \]
\end{proposition}

\begin{proof}
The result follows immediately from~(\ref{Schurpower}),~(\ref{idempconj}) and~(\ref{JMorthcontent}).
\end{proof}

For $m,n\ge 0$, and $\al \vdash n$, consider the $m$-tuple of permutations $(\pi_1,\ld ,\pi_m)$ in $\cS_n$ such that $\pi_1 \cd \pi_m=$ $\si \in \cC_{\al}$. Then $(\pi_1, \ld ,\pi_m)$ is called a {\em factorization} of $\si$. In the case that the subgroup~$\langle \pi_1,\ld ,\pi_m \rangle$ generated by~$\pi_1,\ld ,\pi_m,\si$ acts transitively on $[n]$, we call it a {\em transitive} factorization. Transitive factorizations have been studied extensively in recent years because they encode {\em branched covers} of the sphere by a Riemann surface; in this case we refer to the tuple $(\pi_1,\ld ,\pi_m)$ as the branched cover. Now, the Riemann-Hurwitz Theorem implies that the genus of the covering surface $g \ge 0$ is given by
\begin{equation}\label{RieHur}
n - l(\al) + \sum_{i=1}^m \left( n - l(\cyc (\pi_i) ) \right) = 2n -2 + 2g,
\end{equation}
where $\cyc(\pi_i)$ is the partition whose parts are the lengths of the disjoint cycles in the disjoint cycle representation of the permutation factor $\pi_i$. Equivalently, $\pi_i$ is contained in the conjugacy class $\cC_{\cyc(\pi_i) }$. In this case we say that both the permutation factorization and the corresponding branched cover also have genus $g$.

If each factor $\pi_i = (s_i \, t_i)$, $s_i<t_i$ is a transposition, then we say that the factorization is a {\em transposition} factorization. Note in this case that we have $l(\cyc(\pi_i))=n-1$, $i=1,\ld ,m$, so~(\ref{RieHur}) becomes $m= n + l(\al)-2 + 2g$ for transitive transposition factorizations. If we further constrain a transposition factorization so that $t_1\le \cd \le t_n$, then we call it a {\em monotone} transposition factorization (these have been previously considered, {\em e.g.}, in~\cite{ggn1}).

Next we will use the result of Proposition~\ref{Phiffactninterp} to obtain a description for the series expansion of the logarithm of the content series. The proof involves a combinatorial interpretation in terms of transitive monotone transposition factorizations.
\begin{corollary}\label{logPhif}
For $f(x) = \sum_{i\ge 0} f_i x^i$, where $f_0,f_1,\ld$ are indeterminates, we have
\[ \log \Phi^{f(x)}(y,z,\bp) = \Psi^{f(x)}(y,z,\bp) = \sum_{g\ge 0} \Psi^{f(x)}_g(y,z,\bp) ,\]
where
\[ \Psi^{f(x)}_g(y,z,\bp) = \sum_{n\ge 1} \frac{z^n}{n!} \sum_{\be \vdash n} b(\be ,g) p_{\be} y^{n + l(\be) - 2 + 2g},\]
for some polynomials $b(\be ,g)$ in  $f_0,f_1,\ld$.
\end{corollary}

\begin{proof}
For $m,n\ge 0$, $i_1,\ld ,i_n\ge 0$ with $i_1+\ld +i_n=m$, and $\al \vdash n$, consider monotone transposition factorizations $((s_1 \, t_1),\ld ,( s_m\, t_m))$ $s_i<t_i, i=1,\ld ,m$, in which for $j=1,\ld ,n$, exactly $i_j$ of $t_1,\ld ,t_m$ are equal to $j$ (which means that the first $i_1$ are equal to $1$, the next $i_2$ are equal to $2$, etc., and the last $i_n$ are equal to $n$). If we let $a(\al,m)$ be the sum of $f_{i_1} \cd f_{i_n}$ over all monotone transposition factorizations that correspond to a given choice of $m,n,\al$, then it follows directly from the statement of Proposition~\ref{Phiffactninterp} that we can write
\[ \Phi^{f(x)}(y,z,\bp)  = \sum_{n\ge 0} \frac{z^n}{n!} \sum_{\al \vdash n} \sum_{m\ge 0} a(\al,m) p_{\al} y^m. \]
Equivalently, $\Phi^{f(x)}$ is the generating function for monotone transposition factorizations to which each factorization contributes the monomial $\frac{z^n}{n!} f_{i_1}\cd f_{i_n} p_{\al} y^m$. This is thus an {\em exponential} generating function in $z$, and it is a standard part of the theory of exponential generating functions (see, {\em e.g.},~\cite{gj0}) that the logarithm of an exponential generating function is the generating function for the subset of {\em connected} objects. In the case of monotone transposition factorizations, the connected objects are precisely the {\em transitive} monotone transposition factorizations, and as noted above, it follows from~(\ref{RieHur}) that $m=n+l(\al)-2+2g$ for some $g\ge 0$.  When we apply the logarithm, we thus obtain
\[ \log \Phi^{f(x)} = \Psi^{f(x)} = \sum_{g \ge 0} \sum_{n\ge 1} \frac{z^n}{n!} \sum_{\be \vdash n} b(\be ,g) p_{\be} y^{n+ l(\be) -2 +2g} ,\]
where $b(\be ,g)$ is the sum of $f_{i_1} \cd f_{i_n}$ over all transitive monotone transposition factorizations that correspond to a given choice of $g,n,\al$. The result follows immediately. 
\end{proof}

\section{Special cases of the content series}\label{s4}

We now consider three special cases of the transitive permutation factorizations described in Section~\ref{s3}. Let $H_g(\al)$ be the number of genus $g$ transitive transposition factorizations of $\si$ ranging over all $\si\in\cC_{\al}$, for $\al\vdash n$, $n\ge 1$, and $g\ge 0$. The $H_g(\al)$ are called {\em Hurwitz} numbers, which have been studied extensively in all genera (see, e.g.,~\cite{gj1}) because of their appearance in Witten's Conjecture~\cite{wi}, first proved by Kontsevich~\cite{kon}, and then the moduli space description in~\cite{elsv}. If we define the generating series
\[ \Hu = \Hu(y,z,\bp) = \sum_{g \ge 0} \Hu_g, \qquad \Hu_g = \sum_{n\ge 1}\frac{z^n}{n!} \sum_{\be \vdash n} \frac{H_g(\be)}{( n + l(\be) - 2 + 2g )!} p_{\be} y^{n + l(\be) - 2 + 2g}, \quad g\ge 0, \]
then it was proved in~\cite{gj2} that
\[ \Hu = \log \left( \sum_{n\ge 0} \frac{z^n}{n!} \sum_{\la \vdash n} \chi^{\la}_{(1^n)} s_{\la} (\bp) \prod_{ \Box \in \la} e^{yc(\Box)} \right) . \]
Thus, from Corollary~\ref{logPhif} and~(\ref{defPhifseries}), we have
\begin{equation}\label{caseHu}
\Hu=\Psi^{e^x}.
\end{equation}

As a second special case, let $\vec{H}_g(\al)$ be the number of genus $g$ transitive monotone transposition factorizations of $\si$ ranging over all $\si\in\cC_{\al}$, for $\al\vdash n$, $n\ge 1$, and $g\ge 0$. These numbers $\vec{H}_g(\al)$ are called {\em monotone Hurwitz} numbers, and have been studied in all genera~\cite{ggn1, ggn2}. Among their properties is a close connection to the HCIZ matrix integral (e.g., see~\cite{ggn3}.) If we define the generating series
\[ \vec{\Hu} = \vec{\Hu}(y,z,\bp) = \sum_{g \ge 0} \vec{\Hu}_g, \qquad \vec{\Hu}_g = \sum_{n\ge 1} \frac{z^n}{n!} \sum_{\be \vdash n} \vec{H}_g(\be) p_{\be} y^{n + l(\be) - 2 + 2g}, \quad g \ge 0,\]
then from the proof of Corollary~\ref{logPhif} with $f(x) = \sum_{i\ge 0} 1\cdot x^i=(1-x)^{-1}$, we have
\[ \vec{\Hu} = \log \left( \sum_{n\ge 0} \frac{z^n}{n!} \sum_{\la \vdash n} \chi^{\la}_{(1^n)} s_{\la} (\bp) \prod_{ \Box \in \la} \left( 1 - yc(\Box) \right)^{-1} \right) .  \]
Thus, from Corollary~\ref{logPhif} and~(\ref{defPhifseries}), we have
\begin{equation}\label{casevecHu}
\vec{\Hu} = \Psi^{(1-x)^{-1}}.
\end{equation}

As a third special case, let $G^{(m)}_g(\al)$ be the number of genus $g$ transitive factorizations of $\si$ with precisely $m$ factors, ranging over all $\si\in\cC_{\al}$, for $\al\vdash n$, $n \ge 1$, and $m,g\ge 0$. (Note that in this case, where the factors are arbitrary permutations in $\cS_n$, given $\al$ and $n$, $m$ and $g$ are constrained by~(\ref{RieHur}), but $m$ is not uniquely determined by $g$, unlike in the case where the factors are transpositions.) These numbers $G^{(m)}_g(\al)$ are called $m$-{\em hypermap} numbers, and have been studied in genus $0$, see~\cite{bms,gs,fang}. If we define the generating series
\[ \Hy^{(m)} = \Hy^{(m)}(y,z,\bp) = \sum_{g \ge 0} \Hy^{(m)}_g, \qquad \Hy^{(m)}_g = \sum_{n\ge 1} \frac{z^n}{n!} \sum_{\be \vdash n} G^{(m)}_g(\be) p_{\be} y^{n + l(\be) - 2 + 2g}, \quad g \ge 0, \]
then it was proved in~\cite{gj2} that
\[ \Hy^{(m)} = \log \left(  \sum_{\la \vdash n} \chi^{\la}_{(1^n)} s_{\la} (\bp) \prod_{ \Box \in \la} \left( 1 + yc(\Box)\right)^m \right) .  \]
Thus, from Corollary~\ref{logPhif} and~(\ref{defPhifseries}), we have
\begin{equation}\label{caseHy}
\Hy^{(m)} = \Psi^{(1+x)^m}.
\end{equation}
 
 Thus in all three cases, via~(\ref{caseHu}),~(\ref{casevecHu}) and~(\ref{caseHy}), the generating series are specializations of the content series and its logarithm. 
 
\section{Constructing the $\cU$ and $\cD$ Operators}\label{UDsec}

In this section we give a construction of the $\cU$ and $\cD$ operators in terms of the {\em Bernstein} operators $\{ B_n : n \in \bbZ \}$, which are differential operators acting on symmetric functions (see~\cite[p. 95]{mac} and~\cite[p. 69]{zele}). The generating series for  the Bernstein and adjoint Bernstein operators can be written as
 	\[ B(t) = \sum_{n \in \bbZ} B_n t^n = \exp\left( \sum_{k \geq 1} \frac{t^k}{k} p_k \right) \exp \left( - \sum_{k \geq 1} \frac{t^{-k}}{k} p_k^\perp \right), \]
 	\[ B^\perp(t) = \sum_{n \in \bbZ} B^\perp_n t^n = \exp\left( - \sum_{k \geq 1} \frac{t^k}{k} p_k \right) \exp \left( \sum_{k \geq 1} \frac{t^{-k}}{k} p_k^\perp \right). \]
In the terminology of Section~\ref{sec1}, note that both $B_n$ and $B_n^{\perp}$ change the weight of monomials in the elements of~$\bp$ by $n$, $n \in \bbZ$.

In order to describe the action of these operators in a convenient way, we will define some additional specialized terminology and notation for partitions local to this section. We consider partitions as being drawn on the integer lattice and assign to each cell in the lattice its content, not just the cells in the diagram of the partition. Given a partition $\la$, we define the {\em boundary curve} of $\la$ to be the union of the following curves associated with the diagram of $\la$: (i) the line segments on the bottom edge of cells at the end of a column, (ii) the line segments on the right edge of cells at the end of a row, (iii) the one-way infinite horizontal line extending to the right from the top right corner of the cell at the end of row 1, (iv) the one-way infinite vertical line extending down from the bottom left corner of the cell at the end of column 1. The cells which lie below and/or to the right of the diagram of $\la$ but which intersect the boundary curve of $\la$, are called the {\em outer boundary} of $\la$. For example, on the left hand side of Figure~\ref{Figure.content1} the diagram of the partition $\la = (5,3,3,1)$ appears along with the cells and the corresponding contents in part of the outer boundary of $\la$. Note that each cell in the outer boundary of any partition $\la$ is uniquely determined by its content, and the set of contents of cells in the outer boundary of $\la$ is the set of integers. Given a cell $\square$ in the outer boundary of a partition $\la$ we will say that $\square$ lies to the {\em right} of $\la$ if the left edge of $\square$ is contained in the boundary curve of $\la$. Similarly we will say that $\square$ lies {\em below} $\la$ if the top edge of $\square$ is contained in the boundary curve of $\la$.

\begin{figure}[h]
	\centering
	\includegraphics[width=0.7\textwidth]{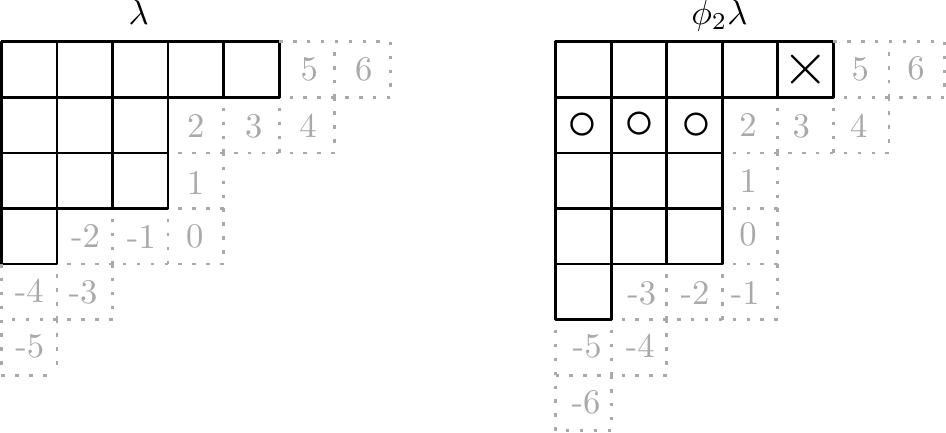}
	\caption{Example of outer boundary of $\la$ and $\phi_c\la$}
	\label{Figure.content1}
\end{figure}

\begin{figure}[h]
	\centering
	\includegraphics[width=0.7\textwidth]{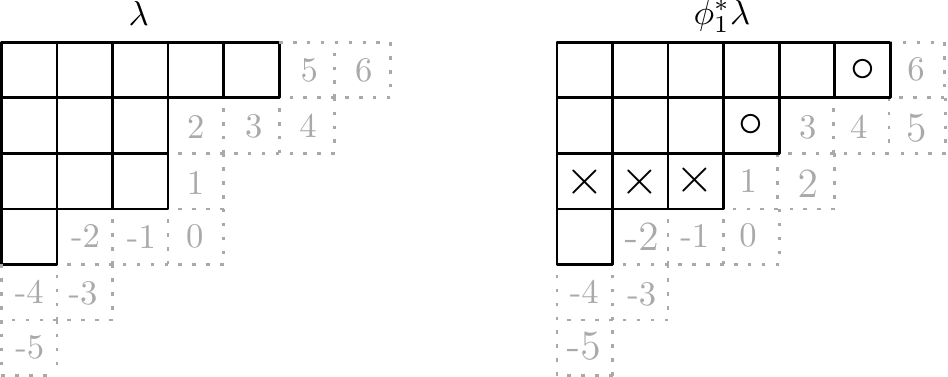}
	\caption{Example of $\phi^*_c\la$}
	\label{Figure.content2}
\end{figure}

Given a partition $\la$ and an integer $c$, if the cell $\square$ in the outer boundary of $\la$ with content $c$ lies below $\la$ then let
$\phi_c \la$ be the partition constructed from $\la$ by removing the last cell in each row of $\la$ ending in a cell with content greater than
$c$ and adding a cell to each column of $\la$ ending in a cell with content less than or equal to $c$. For example, $\phi_2(5,3,3,1)=(4,3,3,3,1)$, as illustrated on the right hand side of Figure~\ref{Figure.content1}. Let $r_c(\la)$ be the number of rows of $\la$ ending in a cell with content greater than $c$, i.e., the number of cells removed when constructing $\phi_c \la$. For example, $r_2(5,3,3,1)=1$. If $\square$ does not lie below $\la$ then leave $\phi_c \la$ as undefined and let $r_c(\la) = 0$.

Similarly, given a partition $\la$ and an integer $c$, if the cell $\square$ in the outer boundary of $\la$ with content $c$ lies to the right of $\la$
then let $\phi^*_c \la$ be the partition constructed from $\la$ by adding a cell to each row of $\la$ ending in a cell with content greater than or equal to $c$ and removing a cell from each column of $\la$ ending in a cell with content less than $c$. For example, $\phi^*_1(5,3,3,1)=(6,4,1)$, as illustrated on the right hand side of Figure~\ref{Figure.content2} (note that to obtain the diagram of $(6,4,1)$, one must slide upwards the cell below the row that has been removed). Let $r^*_c(\la)$ be the number of rows of $\la$ ending in a cell with content greater than or equal to $c$, i.e., the number of cells added when constructing $\phi^*_c \la$. For example, $r^*_1(5,3,3,1)=2$. If $\square$ does not lie to the right of $\la$ then leave $\phi^*_c \la$ as undefined and let $r^*_c(\la) = 0$.

The following result giving the action of $B_n$ and $B^{\perp}_n$ on a Schur function appeared in~\cite{cg}, where it was applied in the context of the KP hierarchy. The statement of the result has been adapted to use the notation described above.

\begin{proposition}\label{Bernstein.Action}
Given a partition $\la$ and an integer $n$, we have
			\[ B_n s_\la = (-1)^{r_n(\la)} s_{\phi_n \la}, \qquad B^\perp_n s_\la = (-1)^{r^*_n(\la)} s_{\phi^*_{-n} \la}, \]
	with the convention that a Schur function whose index is undefined is zero.
\end{proposition}
\begin{proof}
In~\cite[Thm. 3.4]{cg} we proved that $B_n s_\la = (-1)^{i - n - 1} s_{\la^{(i)}}$ where $\la^{(i)} = (\la_1 - 1, \cdots, \la_{u_i(\la)} - 1, i-1, \la_{u_i(\la) + 1}, \cdots)$ with nonnegative integer $u_i(\la)$ chosen uniquely so that $\la_{u_i(\la)} \geq i > \la_{u_i(\la) + 1}$, and positive integer $i$ chosen uniquely so that $|\la^{(i)}| - |\la| = n$. But for any positive integer $i$ it is straightforward to see that $\la^{(i)} = \phi_{i - 1 - u_i(\la)} \la$, $|\la^{(i)}| - |\la| = i - 1 - u_i(\la)$ and $r_{i - 1 - u_i(\la)}(\la)=u_i(\la)$, and the result follows immediately for $B_n s_{\la}$.

In~\cite{cg} we also proved that $B^\perp_n s_\la = (-1)^{j - 1} s_{\la^{(-j)}}$ where $\la^{(-j)} = (\la_1 + 1, \cdots, \la_{j-1} + 1, \la_{j+1}, \cdots)$, and positive integer $j$ is chosen uniquely so that $|\la^{(-j)}| - |\la| = n$. But for any positive integer $j$ it is straightforward to see that $\la^{(-j)} = \phi^*_{\la_j + 1 - j}\la$, $|\la^{(-j)}| - |\la| = j - (\la_j + 1)$ and $r^*_{j-(\la_j+1)}(\la)=j-1$, and the result follows immediately for $B^{\perp}_n s_{\la}$.
\end{proof}

Now that we have a convenient combinatorial description of the action of $B_n$ and $B^{\perp}_n$ on Schur functions, we consider the
action of the compound operator  $B_n B_{-m}^\perp$.

\begin{theorem}\label{RimHook.Action}
	Suppose that $\la$ is a partition and $n$, $m$ are integers such that $B_n B^\perp_{-m} s_\la \not = 0$. Then
		\begin{itemize}
		  \item If $n > m$, then $B_n B^\perp_{-m} s_\la = (-1)^{ht(\mu / \la)} s_\mu$ where $\mu / \la$ is a rim hook of length $n - m$ such that
		        $\bc(\mu / \la) = \{m, m+1, \cdots, n-1 \}$.
		  \item If $m > n$, then $B_n B^\perp_{-m} s_\la = (-1)^{ht(\la / \mu)} s_\mu$ where $\la / \mu$ is a rim hook of length $m - n$ such that
		        $\bc(\la / \mu) = \{n, n+1, \cdots, m-1 \}$.
		  \item If $n = m$, then $B_n B^\perp_{-m} s_\la = s_\la$ and the cell in the outer boundary of $\la$ with content $n$ lies to the right of $\la$.
		\end{itemize} 
\end{theorem}
\begin{proof}
	First, suppose $n > m$. Since $B_n B^\perp_{-m} s_\la \not = 0$ we know that $\mu = \phi_n \phi^*_m \la$ exists. Now the subset of cells added to $\phi^*_m \la$ with content greater than or equal to $n$ coincides with the subset of cells removed
	from $\mu$ with content greater than $n$, and similarly the subset of cells removed from $\phi^*_m \la$ with content less than $m$
	coincides with the subset of cells added to $\mu$ with content less than $m$. This means that the outer boundary of $\la$ and $\mu$
	is the same for cells with content greater than or equal to $n$ or less than $m$.
	
	\begin{figure}[h]
		\centering
		\includegraphics[width=0.85\textwidth]{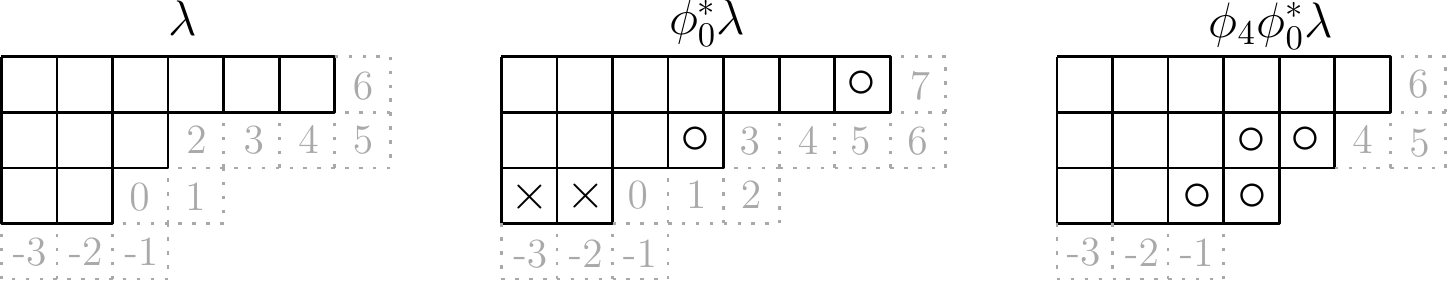}
		\caption{Example $\la$ and $\phi_4 \phi^*_0 \la$}
		\label{Figure.rimHook1}
	\end{figure}
	
	Let $\square$ be a cell in the outer boundary of $\la$ with content in the set $\{m, m+1, \cdots, n-1 \}$. If $\square$ is to the right
	of $\la$, such as the cell in the outer boundary of $\la$ with content $2$ in Figure~\ref{Figure.rimHook1}, then it is added to $\phi^*_m \la$
	and so is in $\mu$. If $\square$ is below $\la$ but not to the right of $\la$, such as the cell in the outer boundary of $\la$ with content $3$
	in Figure~\ref{Figure.rimHook1}, then $\square$
	must be in the outer boundary of $\phi^*_m \la$ and must be added to $\phi_n \phi^*_m \la$ and so is in $\mu$. If $\square$ is not to the
	right of $\la$ or below $\la$, such as the cell in the outer boundary of $\la$ with content $1$ in Figure~\ref{Figure.rimHook1}, then there
	must be a row of $\la$ whose rightmost cell has content equal to the content of $\square$. In this
	case a cell must be added to this row in $\phi^*_m \la$ and so $\square$ lies below $\phi^*_m$ and hence is added to $\mu$. In each case
	$\square$ is contained in $\mu$. Since there are $n - m$ such cells and also since $|\mu| - |\la| = n - m$, the result follows.
	
	Now we have $B_n B^\perp_{-m} s_\la = (-1)^{r_n(\phi^*_m \la) + r^*_m(\la)} s_\mu$. Since each cell removed in $\mu$ must have been added
	to $\phi^*_m \la$ we see that $r_n(\phi^*_m \la) - r^*_m(\la)$ is the number of cells in $\mu$ that were added in $\phi^*_m \la$. Each of these
	cells must necessarily be in their own row of $\mu / \la$ and with the exception of the row in $\mu / \la$ containing the cell with content
	$m$, each such row must have a cell that was added in $\phi^*_m \la$. Thus, the number of cells in $\mu$ that were added in $\phi^*_m \la$ is
	one less than the number of rows in $\mu / \la$, so $r_n(\phi^*_m \la) - r^*_m(\la) = ht(\mu / \la)$, completing the proof for the case $n>m$. 
	
	The case that $m > n$ follows similarly to the above and so we shall not include it. If $n = m$ then by the same argument as above, $\la$ and $\phi_n \phi^*_m \la$ must have the same outer boundary and so $\phi_n \phi^*_m \la = \la$. Also, since $\phi^*_m \la$ is defined, the cell in the outer boundary of $\la$ with content $m$ must be to the right of $\la$, giving the result in the case $n=m$.     
\end{proof}

The following result is an immediate corollary of the above theorem, stated in terms of the generating series $B$ and $B^{\perp}$ for the Bernstein and adjoint Bernstein operators. For a skew partition $\la / \mu$, we use the notation $\si(\la / \mu) = \sum_{\square \in \la / \mu} c(\square) = \sum_{c \in \bc(\la / \mu)} c$.

\begin{corollary}\label{allUD.generatingSeries}
	For any positive integer $k$ and any partition $\la$,
		\[ [t^k] q^{-\frac{1}{2}} B(tq^{\frac{1}{2}}) B^\perp (tq^{-\frac{1}{2}}) s_\la = \sum_\mu (-1)^{ht(\mu / \la)} q^{\frac{\si(\mu/\la)}{k}} s_\mu, \]
		\[ [t^{-k}] q^{-\frac{1}{2}} B(tq^{\frac{1}{2}}) B^\perp (tq^{-\frac{1}{2}}) s_\la = \sum_\mu (-1)^{ht(\la / \mu)} q^{\frac{\si(\la/\mu)}{k}} s_\mu, \]
	where the first sum is over all $\mu$ such that $\mu / \la$ is a rim hook of length $k$ and the second sum is over all $\mu$ such that $\la / \mu$ is a rim
	hook of length $k$. Also,
		\[ [t^0] q^{-\frac{1}{2}} B(tq^{\frac{1}{2}}) B^\perp (tq^{-\frac{1}{2}}) s_\la = \sum_{m=1}^{\ell(\la)} q^{\la_m - m + 1/2} s_\la. \]
\end{corollary}

\begin{proof}
	Extracting coefficients, we have,
		\[ [t^k] q^{-\frac{1}{2}} B(tq^{\frac{1}{2}}) B^\perp (tq^{-\frac{1}{2}}) s_\la = \sum_{n \in \bbZ} q^{n - \frac{k+1}{2}} B_n B_{k-n}^\perp s_\la. \]
	Suppose for some $n \in \bbZ$, $B_n B^\perp_{k-n} s_\la \not = 0$. Then from Theorem~\ref{RimHook.Action} we have 
	$B_n B_{k-n}^\perp s_\la = (-1)^{ht(\mu / \la)} s_\mu,$
	where $\mu / \la$ is a rim hook of length $k$ and where $\bc(\mu / \la) = \{n-k, n-k+1, \cdots, n-1\}$. Now,
	$\si(\mu / \la) = \sum_{c \in \bc(\mu / \la)} c = k(n-k) + \binom{k}{2},$
	so, $\frac{\si(\mu / \la)}{k} = n - k + \frac{k-1}{2} = n - \frac{k+1}{2}$ as required. 
	
The expressions for the coefficients of $t^{-k}$ and $t^0$ follow similarly.
\end{proof}

We are now able to construct the $\cU$ and $\cD$ operators in terms of the Bernstein operators. We express our result in terms of the generating series $\cU(x) = \sum_{k \geq 0} \cU_k \frac{x^k}{k!}$ and $\cD(x) = \sum_{k \geq 0} \cD_k \frac{x^k}{k!}$.

\begin{theorem}\label{UD.generatingSeries}
Let
		\[ V(q, t, \bfp, \bfpp) = \exp\left(\sum_{k \geq 1} (q^\frac{k}{2} - q^{-\frac{k}{2}}) p_k \frac{t^k}{k}\right)
										\exp\left(\sum_{k \geq 1} (q^\frac{k}{2}-q^{-\frac{k}{2}}) p_k^\perp \frac{t^{-k}}{k}\right). \]
Then we have
		\[ \cU(x) = [t^1] \frac{1}{e^{\frac{x}{2}} - e^{-\frac{x}{2}}} V(e^x, t, \bfp, \bfpp), 
							\qquad \cD(x) = [t^{-1}] \frac{1}{e^{\frac{x}{2}} - e^{-\frac{x}{2}}} V(e^x, t, \bfp, \bfpp), \]
\end{theorem}

\begin{proof}
	First note that Corollary~\ref{allUD.generatingSeries} implies that
		\[ \cU(x) = [t^1] e^{-x/2} B(te^{x/2})B^\perp(te^{-x/2}), \qquad \cD(x) = [t^{-1}] e^{-x/2} B(te^{x/2})B^\perp(te^{-x/2}). \]
	
Now, for formal power series $f,g$ in the indeterminate $z$, and $D=a\frac{d}{dz}$ where $a$ is a scalar, Liebniz rule gives the operator identity $e^Df= (e^Df) e^D$. Also, if $b$ is a scalar, then we have $e^D e^{bz}= e^{ab}e^{bz}$. Combining these results using indeterminates $p_1,p_2,\ld$, and noting that $\pp_k p_m=0$ for $k,\neq m$, we obtain
\[ \exp \left( \sum_{k\ge 1} a_k \pp_k \right) \exp \left( \sum_{m\ge 1} b_m p_m \right) = \exp\left( \sum_{k\ge 1} k a_k b_k \right)  \exp \left( \sum_{m\ge 1} b_m p_m \right) \exp \left( \sum_{k\ge 1} a_k \pp_k \right),  \]
for scalars $a_k,b_k,\; k\ge 1$. Thus we have
		\begin{align*}
			q^{-\frac{1}{2}} & B(tq^{\frac{1}{2}}) B^\perp (tq^{-\frac{1}{2}}) \\
					&= q^{-\frac{1}{2}}\exp\left( \sum_{k \geq 1} \frac{q^{\frac{k}{2}} t^{k}}{k} p_k \right)
						\exp\left( - \sum_{k \geq 1} \frac{q^{-\frac{k}{2}} t^{-k}}{k} p_k^\perp \right)
						\exp\left( -\sum_{m \geq 1} \frac{q^{-\frac{m}{2}} t^m}{m} p_m \right)
						\exp\left( \sum_{m \geq 1} \frac{q^{\frac{m}{2}} t^{-m}}{m} p_m^\perp \right) \\
					&=  q^{-\frac{1}{2}} \exp\left(\sum_{k \geq 1} k \frac{q^{-k}}{k^2} \right) V(q, t, \bfp, \bfpp) = \frac{1}{ q^{\frac{1}{2}} - q^{-\frac{1}{2}} } V(q, t, \bfp, \bfpp),
		\end{align*}
		and the result follows with $q=e^x$.
\end{proof}

The series $V$ that appears in Theorem~\ref{UD.generatingSeries} has previously appeared in~\cite{lt}, in connection with a different problem that also concerned Jucys-Murphy elements. The first consequence of this result concerns the form of the operators $\cU_k$ and $\cD_k$ as partial differential operators in the elements of $\bp$.

\begin{corollary}\label{Ukform}
For $k\ge 1$, we have $\cU_k = \sum_{h\ge 0}  \cU_k^{(h)}$, where
\[ \cU_k^{(h)} = \sum_{ \substack{ \ga,\al \in \cP, \\ l(\ga), l(\al) \ge 1, \\ l(\ga) + l(\al) = k+1-2h,\\ | \ga |=| \al |+1 }} \!\!\!\!\!\!\!\! c(\ga,\al)\,p_{\ga}\, \pp_{\al},\]
for some scalars $c(\ga,\al)$.
\end{corollary}

\begin{proof}
{}From Theorem~\ref{UD.generatingSeries} we obtain
\begin{align*}
\cU_k &= [x^k t^1] \frac{\tfrac{1}{2}}{\sinh \tfrac{x}{2}} \exp\left( \sum_{k\ge 1} \frac{\sinh \tfrac{kx}{2}}{\tfrac{k}{2}}p_k t^k\right)
 \exp\left( \sum_{k\ge 1} \frac{\sinh \tfrac{kx}{2}}{\tfrac{k}{2}}\pp_k t^{-k}\right)\\
 &= [x^kt^1] \frac{1}{x\sum_{i\ge 0} \frac{x^{2i}}{2^{2i}(2i+1)!}} \exp\left(\sum_{i\ge 0}\frac{x^{2i+1}}{2^{2i}(2i+1)!}\sum_{k\ge 1} k^{2i}p_k t^k  \right) \exp\left(\sum_{i\ge 0}\frac{x^{2i+1}}{2^{2i}(2i+1)!}\sum_{k\ge 1} k^{2i}\pp_k t^{-k}  \right),
\end{align*}
and the result follows straightforwardly.
\end{proof}
Using the notation of Section~\ref{sec1}, Corollary~\ref{Ukform} says that  $\cU^{(h)}_k$, $h\ge 0$, and thus $\cU_k$ itself, increases the weight of monomials in the elements of~$\bp$ by~$1$.

Consider the generating series for these operators $\cU_k^{(h)}$ given by $\cU^{(h)}(x)=\sum_{k\ge 0} \cU_k^{(h)}\frac{x^m}{m!}$, for $h\ge 0$. To state explicit expressions for these generating series in a compact form, it will be convenient to use some notation: Let
\[ P(t) = \sum_{k\ge 1} p_k t^k,\qquad \Pp(t) = \sum_{k\ge 1} \pp_k t^k, \qquad Q(t)=P(t)+\Pp(t^{-1}),\]
and define $D=t\frac{d}{dt}$, and then $Q_i = D^iQ(t)$, $i\ge 0$. We also introduce the linear operator $\Om$ that shuffles $p_i$'s to the left and $\pp_i$'s to the right in a monomial as if they commute with each other. Then from the last equality in the proof of Corollary~\ref{Ukform} we obtain
\[ \cU^{(h)}(x) = \Om [w^{2h}t^1] \frac{1}{x\sum_{i\ge 0} \frac{w^{2i}x^{2i}}{2^{2i}(2i+1)!}} \exp\left(\sum_{i\ge 0}\frac{w^{2i}x^{2i+1}}{2^{2i}(2i+1)!} Q_{2i}\right), \]
and evaluating this coefficient for small values of $h$ gives (with $Q=Q_0$)
\begin{align}
\cU^{(0)}(x)&= \Om \, x^{-1}e^{xQ},\label{Useries0}\\
\cU^{(1)}(x)&= \Om \tfrac{1}{24}\left( x^2Q_2 -x \right) e^{xQ},\label{Useries1}\\
\cU^{(2)}(x)&= \Om \tfrac{1}{5760} \left( 5x^5Q_2^2 +3x^4Q_4-10x^4Q_2 - 3x^3\right) e^{xQ}. \notag
\end{align}

For example, by evaluating the coefficient of $\tfrac{x^k}{k!}$ in~(\ref{Useries0}), we immediately obtain the following expression for the operators $\cU_k^{(0)}$.

\begin{corollary}\label{Uopsgenus0}
We have $\cU^{(0)}_0=p_1$, and for $k\ge 1$, we have
$$\cU_k^{(0)}= \sum_{j=1}^{k} \frac{1}{k+1}\binom{k+1}{j}\sum p_{a_1}\cd p_{a_j} \pp_{b_1} \cd \pp_{b_{k+1-j}}, $$
where the inner summation ranges over the set $\cT_{j,k+1-j}$ consisting of all positive integers $a_1,\ld ,a_j,b_1,\ld ,b_{k+1-j}$ subject to the restriction that
$$a_1+\ld +a_j = b_1 + \ld +b_{k-j+1} + 1.$$
\end{corollary}

\section{Grading the $\cU$ and $\cD$ operators by genus}\label{s6}

In Corollary~\ref{Ukform}, we proved that $\cU_k$, as a partial differential operator, increases the weight of monomials in the elements of~$\bp$ by~$1$. Moreover, $\cU_k$ could be written as a sum of monomials in the $p_i$'s and $\pp_i$'s of total degree $k+1-2h$ where $h$ is any nonnegative integer. In addition, we gathered together the monomials for each such choice of $h$ into an operator that we called $\cU_k^{(h)}$, $h\ge 0$. In this section we focus on this nonnegative integer $h$, and prove that it actually can be considered as a parameter of genus; so from this point of view the operator $\cU_k^{(h)}$ is the genus $h$ portion of $\cU_k$. To do so, we will again consider the Jucys-Murphy elements that were defined in~(\ref{JMdefn}). The derivation that follows is very similar to the construction of the join-cut operator for star factorizations constructed in \cite{gj3}.

Consider any $n\ge 1$ and any permutation $\si\in\cS_n$ (but depending on the context, we'll also use $\si$ for the corresponding element of the group algebra). Let $\sipr$ denote the permutation in $\cS_{M}$ in which $M = n+1$ is a fixed point, and in which $\sipr(i)=\si(i)$ for all $i=1,\ld ,n$. For any permutation in $\cS_M$, we refer to the cycle containing the maximum element $M$ as the {\em $M$-$\,$cycle} of that permutation, so, e.g., the $M$-$\,$cycle of $\sipr$ has length $1$. The key fact that we use relates the operator $\cU_k$ and the Jucys-Murphy element $J_M$ via
\begin{equation}\label{UkJM}
\cU_k \, p_{\cyc(\si)} = \sum_{\tau\in J^k_{M}\sipr} p_{\cyc(\tau)},
\end{equation}
where $\cyc$ was defined in Section~\ref{s3}. Though we have been unable to find equation~(\ref{UkJM}) stated explicitly in this form, it seems clear that it is well known to a number of authors, e.g., for closely related work see the original papers of Jucys~\cite{ju} and Murphy~\cite{mu}, as well as Diaconis and Greene~\cite{dg}), F\'eray~\cite{f1}, Lascoux and Thibon~\cite{lt}, Lassalle~\cite{l1} and Okounkov~\cite{ok}.

Thus we will consider $J^k_{M}\sipr$, and specifically the product
\begin{equation}\label{JMproduct}
\tau = \ka_1 \cd \ka_k\;\sipr,\qquad \qquad \ka_i=(c_i\; M), \quad 1\le c_i \le M-1, \quad i=1,\ld ,k.
\end{equation} 
We show below that the genus of the corresponding branched cover $(\ka_k,\ld ,\ka_1,\sipr)$ (as discussed in Section~\ref{s3}) is precisely $h$, and that it contributes a monomial of total degree $k+1-2h$ in the $p_i$'s and the $\pp_i$'s to $\cU_k$, thus establishing the genus interpretation of the parameter $h$ in $\cU_k^{(h)}$. To do so, we will consider the product~(\ref{JMproduct}) iteratively, and hence define the partial products $\si^{(0)}=\sipr$, $\;\si^{(i)}= \ka_i \;\si^{(i-1)}$, $i=1,\ld ,k$, so in particular $\tau = \si^{(k)}$. Let the length of the $M$-$\,$cycle in $\si^{(i)}$ be $s_i\ge 1$, for $i=0,\ld ,k$. We also define the branched covers $\cB^{(i)}= (\ka_i,\ld ,\ka_1,\si)$, for $i=0,\ld ,k$. From~(\ref{UkJM}), to determine the action of the operator $\cU_k$, we need to determine the lengths of the cycles in the disjoint cycle representation of $\si^{(k)}$ in terms of the length of the cycles in $\si$. Throughout our iteration, we will refer to the cycles of $\si$ as {\em initial} cycles.
\vspace{.1in}

\noindent
\textbf{Iteration:} At stage $i$, for $i=0,\ld ,k$, we have the permutation $\si^{(i)}$, and we introduce a subset $\cH^{(i)}$ of the stages that we consider to be \emph{marked}, for purposes that will become clear; at stage $0$, $\cH^{(0)}$ is the empty set, and we will iteratively update $\cH^{(i)}$ as we proceed through the stages. At stage $0$, the $M$-$\,$cycle of $\si^{(0)}$ has length $s_0=1$, and the other cycles are the initial cycles. Thus, if $k=0$, we immediately have $\cU_0 = \cU^{(0)}_0=p_1$. If $k\ge 1$, then for $i=1,\ld ,k$, at stage $i$ there are two possibilities for the product $\si^{(i)}= \ka_i \; \si^{(i-1)}$: either (a)  $c_i$ is contained in the $M$-$\,$cycle of $\si^{(i-1)}$, or (b) it is not. For (a), then in the product the $M$-$\,$cycle of $\si^{(i-1)}$ is {\em cut} into two cycles - the $M$-$\,$cycle of $\si^{(i)}$ and another cycle, of length $d_i\ge 1$, which we will refer to as a \emph{spare} cycle. For (b), then in the product the $M$-$\,$cycle of $\si^{(i-1)}$ is {\em joined} to another cycle of $\si^{(i-1)}$, of length $e_i\ge 1$, to create the $M$-$\,$cycle of $\si^{(i)}$. Note that for (a) we have $d_i=s_{i-1}-s_i\ge 1$, and for (b) we have $e_i=s_i-s_{i-1}\ge 1$. Thus $\si^{(i)}$ has an $M$-\,cycle together with a (possibly empty) collection of spare cycles (created at some of stages $1,\ld ,i$ by cuts from the $M$-\,cycle), and all other cycles (if any) are initial cycles - more formally the latter initial cycles are all the cycles of $\si$ whose elements are not contained in the orbit of $M$ in $\langle \ka_i,\ld ,\ka_1,\si \rangle$.

Now we go back to case (b) of the iteration described above, and consider two subcases: either (i) the $M$-$\,$cycle of $\si^{(i-1)}$ is joined to an initial cycle of $\si^{(i-1)}$, or (ii) the $M$-$\,$cycle of $\si^{(i-1)}$ is joined to a spare cycle of $\si^{(i-1)}$. In subcase (i), then the genus of the branched   cover $\cB^{(i)}$ is the same as the genus of the branched cover $\cB^{(i-1)}$; in subcase (ii), the genus of the branched cover $\cB^{(i)}$ is $1$ greater than the genus of the branched cover $\cB^{(i-1)}$. That is, in subcase (ii) we are simply rejoining a cycle that was cut from the $M$-\,cycle at a previous stage $i^{\prime}$ to the $M$-\,cycle (though the join point may be different from the cut point). In subcase (ii) we update $\cH^{(i)}$ by $\cH^{(i)} = \cH^{(i-1)} \cup \{ i^{\prime},i \}$, and note that
\begin{equation}\label{produnmarked}
\cyc(\si^{(i)}) = \cyc( ( \!\!\!\!\prod_{\substack{ j\in [i],\\ j \notin \{ i^{\prime},i \} } } \!\!\!\! \ka_j ) \, \sipr ), \qquad\qquad \mbox{ and } e_i = d_{i^{\prime}}.
\end{equation}
Finally, we regard the $M$-$\,$cycle of $\si^{(k)}$ as a cycle that is created  in our iteration at an additional, artificial, $k+1$st stage, and thus define $s_{k+1} =0$, $d_{k+1}=s_{k}$. We then have $d_i=s_{i-1}-s_i\ge 1$ for $i=k+1$.
\vspace{.1in}

At the termination of our iteration, we can conclude from~(\ref{produnmarked}) by induction on the stages that
\[ \cyc(\si^{(k)}) = \cyc( ( \!\!\!\!\prod_{\substack{ j\in [k],\\ j \notin \cH^{(k)} } } \!\!\!\! \ka_j ) \, \sipr ),\]
and that if $|\cH^{(k)}|=2h$, then the genus of $\cB^{(k)}$ is $h$. But in this case we create a spare cycle at exactly $j$ of the $k+1-2h$ unmarked stages, and let the values of the corresponding $d_i$'s be given by $a_1,\ld ,a_j$; we also join an initial cycle to the $M$-$\,$cycle at the remaining $k+1-j-2h$ of the unmarked stages, and let the values of the corresponding $e_i$'s be given by $b_1,\ld ,b_{k+1-j -2h}$. Note that $a_1,\ld ,a_j,b_1,\ld ,b_{k+1-j-2h}\ge 1$, and
\[ \left( b_1 + \ld +b_{k+1-j-2h}\right) -  \left( a_1+\ld +a_j\right) = \sum_{ i\in[k+1]\setminus \cU^{(k)} }(s_{i}-s_{i-1})= \sum_{i\in[k+1]}(s_{i}-s_{i-1})= s_{k+1}-s_{0} = -1. \]
Thus in this case $\si^{(k)}$ contributes the monomial operator $p_{a_1}\cd p_{a_j} \pp_{b_1} \cd \pp_{b_{k+1-j-2h}}$ (multiplied by some scalar coefficient) to the operator $\cU^{(h)}_k$, and comparison with Corollary~\ref{Ukform} gives us the desired conclusion, that from the Jucys-Murphy/branched cover point of view, $\cU^{(h)}_k$ is the genus $h$ portion of $\cU_k$.

For the particular case of genus $0$ we can determine the scalar coefficients of the monomial operators in $\cU^{(0)}_k$ by a more detailed analysis, which gives us the following combinatorial proof of Corollary~\ref{Uopsgenus0}.
\vspace{.1in}

\noindent
\textbf{Combinatorial proof of Corollary~\ref{Uopsgenus0}:} We consider only the products $\si^{(k)}$ that have no marked stages in the above iteration, and thus for which the M-$\,$cycle is only ever joined to initial cycles. Thus as above we create a cycle in the iteration above at exactly $j$ of the $k+1$ stages, where $1\le j\le k$, and let the values of the corresponding $d_i$'s be given by $a_1,\ld ,a_j$. Then we join a cycle to the $M$-$\,$cycle at the remaining $k+1-j$ stages, and let the values of the corresponding $e_i$'s be given by $b_1,\ld ,b_{k+1-j}$. As above in the case $h=0$ we have $a_1,\ld ,a_j,b_1,\ld ,b_{k+1-j}\ge 1$, and
 \[ \left( b_1 + \ld +b_{k-j+1}\right) -  \left( a_1+\ld +a_j\right) = \sum_{i=1}^{k+1} (s_{i}-s_{i-1})= s_{k+1}-s_{0} = -1. \]
Note also that for $i=1,\ld ,k$, the $i$ th partial sum of the ordered $-a$'s and $b$'s is given by $s_i-s_0=s_i-1\ge 0$. But there are $\binom{k+1}{j}$ possible ordered stages with $j$ $d_i$'s and $k+1-j$ $e_i$'s, and the action on the product $p_{\cyc(\si)}$ in genus $0$ is described by the operator
\begin{equation}\label{opwocondn}
\sum_{j=1}^{k} \binom{k+1}{j}\sum p_{a_1}\cd p_{a_j} \pp_{b_1} \cd \pp_{b_{k+1-j}}, 
\end{equation} 
where the inner summation ranges over $\cT_{j,k+1-j}$, if we allow all possible orders. However, we have to account for the condition that all of the partial sums of the $-a$'s and $b$'s are nonnegative.

Now, for any $m$-tuple of integers ${\bf q}=(q_1,\ld ,q_m)$ let the $i$th cyclic shift of ${\bf q}$ be given by ${\bf q}^{(i)}=(q_{i+1},\ld ,q_m,$ $q_1,\ld ,q_i)$, for $i=1,\ld ,m-1$. Then if the sum of the elements of ${\bf q}$ is $-1$, it is well known (by what is often referred to as the {\em Cycle Lemma}) that among the $m$ $m$-tuples given by ${\bf q}$ together with its $m-1$ cyclic shifts, all of the $m$-tuples are distinct, and, moreover, exactly one of the $m$-tuples has partial sums that are all nonnegative. Thus, in the operator~(\ref{opwocondn}), to account for the nonnegative partial sum condition on the $k+1$-tuples in the summation set, we must divide by $k+1$, completing our combinatorial proof of Corollary~\ref{Uopsgenus0}.
\hfill$\square$

\section{A partial differential equation for the content series in genus $0$}\label{s8}
\subsection{The partial differential equation in genus $0$}

In the following result we give a partial differential equation for the genus $0$ portion of the content series. In the terminology of Section~\ref{sec1}, note that on the left hand side the differential operators increase the weight of monomials in the elements of~$\bp$ by~$1$, while on the right hand side, this weight is preserved. The increase in weight on the left hand side reflects the combinatorial fact described in Section~\ref{s6}, that after inserting $n+1$ as a fixed point, we have moved from the symmetric group $\cS_n$ to $\cS_{n+1}$ in the multiplication by the Jucys-Murphy element $J_{n+1}$.
\begin{theorem}\label{genus0pde}
Let $\whPsi^{f(x)}_0 = \whPsi^{f(x)}_0(z,\bp) =\Psi^{f(x)}_0(1,z,\bp)$. Then $\whPsi^{f(x)}$ is the unique solution to the partial differential equation
\begin{equation}\label{whPsipde}
z p_1+ z\sum_{k\ge 1} f_k \sum_{j=1}^k \binom{k+1}{j} \sum_{\ell\ge 0}\sum_{\substack{a_1,\ld ,a_j\ge 1\\a_1+\ld +a_j=\ell+1}} \!\!\!\!\!\! p_{a_1}\cd p_{a_j} [u^{\ell}] \left(  \sum_{i\ge 1} \left( \pp_i \whPsi^{f(x)}_0 \right) u^i  \right)^{k+1-j} \!\!\!\!\!\!\!\! =   \sum_{i\ge 1}p_i \pp_i \whPsi^{f(x)}_0
\end{equation}
that satisfies the initial condition $\whPsi^{f(x)}_0(0,\bp)=0$.
 \end{theorem} 
 
 \begin{proof}
 We denote the set of partitions of $[m]$ into $i$ nonempty subsets $B= \{ B_1,\ld ,B_i\}$ by $\Pi_{m,i}$. The sets $B_j$ are called the {\em blocks} of $B$. For an integer partition $\al=(\al_1,\ld ,\al_m)$ and a subset $\be\subseteq [m]$ with elements $\be_1< \cd <\be_i$, define $\al(\be)$ to be the partition $\{ \al_{\be_1},\ld ,\al_{\be_i}\}$. Now if $\al$ is a partition with $m$ parts and $F=F(\bp)$ is a formal power series in $\bp$, then by the product rule, we have
 \[   \pp_{\al} e^F = \sum_{i=1}^m \sum_{\substack{\{ B_1,\ld ,B_i\}\\ \in \Pi_{m,i}}} \left( \pp_{\al(B_1)} F  \right) \cd \left( \pp_{\al(B_i)} F  \right) e^F .\]
Now consider this result together with Corollaries~\ref{Phifpde},~\ref{logPhif},~\ref{Ukform} and~\ref{Uopsgenus0}. Note that if we multiply equation~(\ref{fpde}) on both sides by $ze^{-\Psi^{f(x)}}$ (after carrying out the differentiations), then on the LHS we have a linear combination of monomials in $z,y,\bp$ - one of these monomials is $z p_1$, and all other monomials are of the form
\[ S=  z \, y^k \, p_{\ga} \left( \pp_{\al(B_1)} z^{ | \nu_1 | } \, p_{\nu_1} \,y^{ | \nu_1 | + l( \nu_1 ) - 2 + 2g_1 } \right) \cdots  \left(  \pp_{\al(B_i)} z^{ | \nu_i | } \, p_{\nu_i} \, y^{ | \nu_i | + l( \nu_i ) - 2 + 2g_i } \right), \]
where $k \ge 1$, $l(\al)=m$, $m\ge 1$, $l(\ga)= k+1-2s-m\ge 1$, $s \ge 0$, $| \ga |=| \al |+1$, $1\le i\le m$, $| B_1 |,\ld ,| B_i | \ge 1$, $B_1 \sqcup \ld \sqcup B_i=[m]$, $l(\nu_1),\ld ,l(\nu_i)\ge 1$, and $g_1,\ld ,g_i\ge 1$. Thus we have (ignoring the scalar multiple induced by applying the $\pp_i$'s)
\begin{equation}\label{LHSmonom}
S= z^R \, p_{\eta}\, y^N,
\end{equation}
where 
\begin{align*}
| \eta |&= | \ga | + | \nu_1 | + \ld + | \nu_i | - | \al(B_1) | - \ld - | \al(B_i) | , \\
&= | \ga | + | \nu_1 | + \ld + | \nu_i | - | \al | = | \nu_1 | + \ld + | \nu_i | + 1,\\
l( \eta )&=   l( \ga ) + l( \nu_1 ) + \ld + l( \nu_i ) - l( \al(B_1) ) - \ld - l( \al(B_i) ) , \\
&= l( \ga ) + l( \nu_1 ) + \ld + l( \nu_i ) - l( \al ) =  l( \nu_1 ) + \ld + l( \nu_i ) + k + 1 - 2s - 2m, 
\end{align*}
and thus we have
\[ R = 1 + | \nu_1 | + \ld + | \nu_i | = | \eta | , \]
and
\begin{align}
N&= k +  | \nu_1 | + \ld + | \nu_i | +  l( \nu_1 ) + \ld + l( \nu_i ) - 2i + 2g_1 + \ld + 2g_i, \notag \\
&= k + \left( | \eta | - 1 \right) + \left( l( \eta ) - k - 1 + 2s + 2m \right)  - 2i +  2g_1 + \ld + 2g_i , \notag \\
&= | \eta | + l( \eta ) - 2 + 2 \left( s + m - i + g_1 + \ld + g_i  \label{Nformula} \right)  ,
\end{align}
But if we multiply equation~(\ref{fpde}) on both sides by $ze^{-\Psi^{f(x)}}$, then on the RHS we have a linear combination of monomials of the form
\begin{equation}\label{RHSmonom}
z \frac{\partial}{\partial z} z^{ | \eta | } \, p_{\eta} \, y^{ | \eta | + l( \eta ) - 2 + 2g } =   | \eta | \cdot z^{ | \eta | } \, p_{\eta} \, y^{ | \eta | + l( \eta ) - 2 + 2g }, \qquad g\ge 0.
\end{equation}  

Now, suppose that on both sides we consider only monomials of the form $z^{ | \eta | } \, p_{\eta} \, y^{ | \eta | + l( \eta ) - 2}$, for $\eta \in \cP$ with $l(\eta) \ge 1$. Then on the RHS, from~(\ref{RHSmonom}), we must have $g=0$. On the LHS, from~(\ref{LHSmonom}) and~(\ref{Nformula}), we must have $s=g_1= \ld =g_i=0$ and $i=m$, so $| \ga_1 | = \ld = | \ga_i | = 1$. Thus we obtain equation~(\ref{whPsipde}), and the initial condition follows immediately, to complete the proof of the result.
 \end{proof} 
 
 We can obtain pde's that are similar to~(\ref{whPsipde}) for genus up to and including $g$ for each $g\ge 0$, by using the explicit expressions for the $\cU_k^{(g)}$ given in Section~\ref{UDsec}, but they quickly become complicated, and we have not been able to apply such an equation even in the case of genus $1$ to count branched covers in the three special cases.
\subsection{Technical results on symmetric functions for applications of the differential equation} 

In this section we give some technical results for symmetric functions in a finite set of variables $x_1,\ld ,x_n$, that we will use in Section~\ref{s9} to apply Theorem~\ref{genus0pde} for particular cases of the content series. The \emph{complete} symmetric functions are given by $h_0(x_1,\ld ,x_n)=1$, and
\begin{equation}\label{compsymfn}
h_k(x_1,\ld ,x_n) = \sum_{\substack{i_1,\ld ,i_n\ge 0\\ i_1+\ld +i_n=k}}x_1^{i_1}\cd x_n^{i_n},\qquad k\ge 1.
\end{equation}
These are the \emph{Schur} symmetric function $s_{\la}(x_1,\ld ,x_n)$ in the case that the partition $\la$ has the single part $k$. In general, for
a partition $\la=(\la_1,\ld ,\la_n)$ with at most $n$ parts, the \emph{Schur} symmetric function $s_{\la}(x_1,\ld ,x_n)$ can be written as a ratio of determinants via 
\begin{equation}\label{ratioalt}
s_{\la}(x_1,\ld ,x_n)= \frac{\det \left( x_j^{\la_i+ n-i} \right)_{1\le i,j\le n}}{\det \left(  x_j^{n-i} \right)_{1\le i,j\le n}}.
\end{equation}

\begin{proposition}\label{pfcnhm}
For $n\ge 1$, we have
\begin{equation*}
\sum_{i=1}^n x_i^{k+n-1} \prod_{\substack{j=1\\j\ne i}}^n \frac{1}{x_i-x_j} = 
\begin{cases}
h_{k} (x_1,\ld ,x_n) ,\quad k\ge 0,\\
0, \qquad \qquad -(n-1) \le k \le -1.
\end{cases}
\end{equation*}
\end{proposition}

\begin{proof}
The determinant in the denominator of~(\ref{ratioalt}) is the \emph{Vandermonde} determinant, which can be evaluated as the simple product
\[ \det \left(  x_j^{n-i} \right)_{1\le i,j\le n}=\prod_{1\le i< j\le n} (x_i-x_j). \]
Now in the case that the partition $\la$ has the single part $k$, we obtain from~(\ref{ratioalt}) that
\[ h_k(x_1,\ld ,x_n) = \frac{\det \left( x_j^{k\de_{i,1} +n - i} \right)_{1\le i,j\le n}}{\det \left(  x_j^{n-i} \right)_{1\le i,j\le n}} = \sum_{i=1}^n x_i^{k+n-1} \prod_{\substack{j=1\\j\ne i}}^n \frac{1}{x_i-x_j},   \]
where the second equality follows from the cofactor expansion of the numerator determinant in row $1$, and the Vandermonde determinant evaluation given above. The result follows immediately for $k\ge 0$. For $k= 1- m$ and $m$ between $2$ and $n$ inclusive, note that the numerator determinant has row $1$ equal to row $m$, and thus has value $0$, giving the result for $-(n-1) \le k \le -1$.
\end{proof}

For a set of positive integers $\al=\{ \al_1,\ld ,\al_n\}$ of size $n$, $n\ge 1$, and nonnegative integers $a_1,\ld ,a_n$, we define the symmetrization action
\[ \Th_{\al} \; p_{a_1}\cd p_{a_n} = \sum_{\si \in \cS_n} x_{\al_1}^{a_{\si(1)}}\cd x_{\al_n}^{a_{\si(n)}},\]
extended linearly to all polynomials of total degree at most $n$ in $\{ p_i \}_{i\ge 1}$ (where we set $p_0=1$). We will use~$\Th$ to prove that two power series are identical via the following result.

\begin{proposition}\label{identsym}
If $P_1$ and $P_2$ are polynomials of degree at most $n$ in $\{ p_i y^i \}_{i\ge 1}$, for any formal power series $y$ with $0$ constant term, and $\Th_{[n]}\; P_1 = \Th_{[n]}\; P_2$, then $P_1=P_2$.
\end{proposition}

The following result  that relates~$\Th$ to the complete symmetric functions via Proposition~\ref{pfcnhm} is applied later in the paper to determine the generating series for both Hurwitz numbers and $m$-hypermap numbers in genus $0$.

\begin{lemma}\label{umbcomp}
\begin{enumerate}
\item
Consider formal power series $f(u)$ and $g(u)$, in which the coefficients are constant and linear, respectively, in the $\{p_is^i\}_{i\ge 1}$. Then we have
\begin{align*}
\Th_{[n]} \; \sum_{j=1}^n \sum_{\ell\ge 0} \sum_{\substack{a_1,\ld ,a_j\ge 1\\a_1+\ld +a_j=\ell+1}} & \!\!\!\!\!\! \frac{p_{a_1}\cd p_{a_j}}{j!} s^{\ell+1}[u^{\ell}]  \left. f(u) \, \frac{g(u)^{n-j}}{(n-j)!} \right|_{s=1} \\
&= \sum_{i=1} ^n x_i f(x_i) \prod_{\substack{j=1\\j\ne i}}^n \left( \frac{x_j}{x_i - x_j} + \left. \Th_{\{ j \} } g(x_i) \right|_{s=1} \right) .
\end{align*}
\item
Consider a formal power series $g(u)$ in which the coefficients are linear in the $\{p_is^i\}_{i\ge 1}$. Then we have
\begin{align*}
 \Th_{[n]} \; \sum_{k=1}^n \frac{1}{(n-k)!} \sum_{j=1}^k  \sum_{\ell\ge 0} \sum_{\substack{a_1,\ld ,a_j\ge 1\\a_1+\ld +a_j=\ell+1}} & \!\!\!\!\!\! \frac{p_{a_1}\cd p_{a_j}}{j!} s^{\ell+1}[u^{\ell}]  \left. \frac{g(u)^{k-j}}{(k-j)!} \right|_{s=1} \\
&= \sum_{i=1} ^n x_i \prod_{\substack{j=1\\j\ne i}}^n \left( 1 +  \frac{x_j}{x_i - x_j} + \left. \Th_{\{ j \} } g(x_i) \right|_{s=1} \right) .
\end{align*}
\end{enumerate}
\end{lemma}

\begin{proof}
For part (1), let $b(u)=\sum_{\ell\ge 0} b_{\ell} u^{\ell}$ be a formal power series independent of $\{p_is^i\}_{i\ge 1}$, and note that
\begin{align*}
& \sum_{\ell\ge 0} \sum_{\substack{a_1,\ld ,a_j\ge 1\\a_1+\ld +a_j =\ell+1}} \!\!\!\!\!\! \Bigg( \Th_{[j]} p_{a_1}\cd p_{a_j} \Bigg) \left. \!\! [u^{\ell}] b(u) \right|_{s=1} = \; j! \sum_{\ell\ge 0}b_{\ell} \,  x_1 \cd x_j \, h_{\ell + 1 - j}(x_1,\ld ,x_j) \\
& = j! \sum_{\ell\ge 0}b_{\ell} \,  x_1 \cd x_j \sum_{i=1}^j x_i^{\ell} \prod_{\substack{r=1\\r \ne i}}^n \frac{1}{x_i-x_r} = \; j! \sum_{i=1}^j x_i b(x_i) \prod_{\substack{r=1\\r \ne i}}^n \frac{x_r}{x_i-x_r},
\end{align*}
where the second last equality is obtained from Proposition~\ref{identsym} with $n=j$, and $k= \ell + 1 - j$. Thus from the above we obtain
\begin{align*}
LHS &= \sum_{j=1}^n \sum_{\substack{\al \sqcup \be =[n]\\ | \al | = j \\ | \be | = n - j}}  \sum_{\ell\ge 0} \!\! \sum_{\substack{a_1,\ld ,a_j\ge 1\\a_1+\ld +a_j=\ell+1}} \!\!\!\!\!\! \Bigg( \Th_{ \al } \frac{ p_{a_1}\cd p_{a_j}}{j!} \Bigg) s^{\ell+1}[u^{\ell}]  \left. f(u) \Bigg( \Th_{ \be } \frac{g(u)^{n-j} }{(n-j)!}\Bigg) \right|_{s=1} \\
& = \sum_{\substack{\al \sqcup \be =[n]\\ | \al | \ge 1 \\ | \be | \ge 0 }} x_i f(x_i) \Bigg(  \prod_{\substack{r \in \al \\ r \ne i}}^n \frac{x_r}{x_i-x_r}  \Bigg)  \Bigg(  \prod_{b \in \be}^n \Th_{ \{ j \} } \left.  g(x_i) \right|_{s=1} \Bigg) \\
& = \sum_{i=1}^n x_i f(x_i) \!\!\!\! \sum_{\ga \sqcup \be =[n]} \!\! \Bigg(  \prod_{ r \in \ga }^n \frac{x_r}{x_i-x_r}  \Bigg)  \Bigg(  \prod_{b \in \be}^n \Th_{ \{ j \} } \left.  g(x_i) \right|_{s=1} \Bigg) ,
\end{align*}
where for the last equality we have reordered the summations and changed the summation variable $\al$ to $\ga=\al\setminus i$, and for the second last equality, we have applied part~(1) of the result. Part~(2) of the result follows immediately.
\vspace{.05in}

For part (2), we have
\begin{align*}
LHS &= \sum_{k=1}^n \sum_{\substack{\al \sqcup \be =[n]\\ | \al | = n - k\\ | \be | = k}}  \Bigg( \Th_{ \al } \frac{1}{(n-k)!} \Bigg) \Bigg( \Th_{\be} \sum_{j=1}^k \sum_{\ell\ge 0} \!\! \sum_{\substack{a_1,\ld ,a_j\ge 1\\a_1+\ld +a_j=\ell+1}} \!\!\!\!\!\! \frac{ p_{a_1}\cd p_{a_j}}{j!} s^{\ell+1}[u^{\ell}]  \left. \frac{g(u)^{n-j} }{(n-j)!} \right|_{s=1} \Bigg) \\
&= \sum_{\substack{\al \sqcup \be =[n]\\ | \al | \ge 0 \\ | \be | \ge 1}} \Bigg( \prod_{a \in \al} 1 \Bigg)  \sum_{i \in \be} x_i \prod_{b \in \be} \left( \frac{x_j}{x_i - x_j} + \left. \Th_{\{ b \} } g(x_i) \right|_{s=1} \right) \\
&= \sum_{i=1}^n x_i \sum_{\al \sqcup \ga =[n]} \Bigg( \prod_{a \in \al} 1 \Bigg) \prod_{r \in \ga} \left( \frac{x_j}{x_i - x_j} + \left. \Th_{\{ b \} } g(x_i) \right|_{s=1} \right) ,
\end{align*}
where for the last equality we have reordered the summations and changed the summation variable $\be$ to $\ga=\be\setminus i$. The result follows immediately.  
\end{proof}
 
\section{Applications of the content series in genus $0$}\label{s9}

In this section, we apply Theorem~\ref{genus0pde} for the three special cases of the content series identified in Section~\ref{s4}, to obtain new and uniform proofs for the explicit numbers of branched covers in genus $0$. In each case we give an {\em algebraic} proof that the generating series for the appropriate numbers satisfies the genus $0$ pde given by Theorem~\ref{genus0pde}, and that it has the appropriate initial condition (the latter is immediate in each case).

However, in Section~\ref{s6} we have given a {\em combinatorial} proof for the genus $0$ portion of the $\cU$ operators, in terms of multiplication by the Jucys-Murphy element $J_{n+1}$, after inserting a fixed point to move from $\cS_n$ to $\cS_{n+1}$. This gives an underlying combinatorial flavour to our algebraic proof. In the case of Hurwitz numbers, this provides a significant contrast with the proof given in~\cite{gj1}, where a much simpler genus $0$ pde (the join-cut equation) was obtained via a combinatorial analysis of the effect on disjoint cycle lengths when multiplying  within $\cS_n$ by the set of transpositions. In the case of $m$-hypermap numbers, a similar proof could be obtained if we were able to carry out the combinatorial analysis of the effect on disjoint cycle lengths when multiplying within $\cS_n$ by the set of all permutations. We have been unable to carry out this analysis in general, and thus have been unable to obtain such a proof for the genus $0$ $m$-hypermap numbers (but for the case of multiplication by a cycle of arbitrary length, together with fixed points, see~\cite{gj4}). Thus, by moving from~$\cS_n$ to~$\cS_{n+1}$, and {\em not} remaining within $\cS_n$, it seems that we have obtained significant underlying combinatorial advantage in this paper, and we believe that further work along these lines would be worthwhile.
 
\subsection{Hurwitz numbers in genus $0$}\label{s91}

Define the series
\begin{equation}\label{H0explicit}
G = G(z,\bp) = \sum_{n\ge 1} z^n \sum_{\al\vdash n} \frac{p_{\al}}{| \Aut \al |} \, n^{l(\al)-3} \prod_{j=1}^{l(\al)} \frac{\al_j^{\al_j}}{\al_j!} ,
\end{equation}
and recall from Section~\ref{s4} that $\Hu_0(1,z,\bfp)$ is the generating series for Hurwitz numbers in genus $0$. Then it has previously been established (e.g., see~\cite{gj1},~\cite{h}) that $\Hu_0 (1,z,\bp)=G(z,\bp)$. We now give a new proof of this result by applying Theorem~\ref{genus0pde}. To do so, note that from Theorem~\ref{genus0pde} and~(\ref{caseHu}), the generating series $\Hu_0(1,z,\bfp) = \whPsi^{e^x}_0$ satisfies the partial differential equation
\begin{equation}\label{hitzpde}
A(\Hu_0 (1,z,\bp))=0
\end{equation}
where
\begin{align*}
A(F)&= z\sum_{k\ge 1} \frac{1}{(k+1)!} \sum_{j=1}^k \binom{k+1}{j} \sum_{\ell\ge 0}\sum_{\substack{a_1,\ld ,a_j\ge 1\\a_1+\ld +a_j=\ell+1}} \!\!\!\!\!\! p_{a_1}\cd p_{a_j} [u^{\ell}] \left(  \sum_{i\ge 1}(\pp_i F) u^i  \right)^{k+1-j} \!\!\!\!\!\!\!\!\!\!\!\! + p_1 z - \sum_{i\ge 1}p_i \pp_i F.
\end{align*}

Now let $s$ be given by the functional equation $s=z\exp \phi_0(s)$, where $\phi_i(s) =\sum_{j\ge 1}\frac{j^{j+i}}{j!}p_j s^j$, for any integer~$i$. In \cite{gj1} we proved that
\[  \pp_i G =\frac{i^{i-1}}{i!} s^i - \frac{i^{i}}{i!} s^i \sum_{j\ge 1} \frac{j^{j+1}}{j!}p_js^j\frac{1}{i+j}, \]
and thus we have
\begin{equation}\label{ppG}
\sum_{i\ge 1}p_i \pp_i G = \phi_{-1}(s) - \tfrac{1}{2} \phi_0(s)^2, \qquad\qquad\qquad \sum_{i\ge 1}(\pp_i G) u^i = T(us) - S(us),
\end{equation}
where $T(x)=\sum_{i\ge 1}\frac{i^{i-1}}{i!} x^i$ and $S(x)= \sum_{i\ge 1} \frac{i^{i}}{i!} x^i \sum_{j\ge 1} \frac{j^{j+1}}{j!}p_js^j\frac{1}{i+j}$. Then it is well known that $i^{i-1}$ counts the number of rooted labelled trees on $i$ vertices, and that the \emph{tree series} $T(x)$ satisfies the functional equation $T=x e^T$. From the proof of Proposition~3.2 in~\cite{gj1}, we obtain immediately that
\begin{equation}\label{Gsum}
\Th_{\{ j\} } \left. S(u) \right|_{s=1} =  \frac{u}{u-x_j} - \frac{T(u)}{T(u)-T(x_j)} \frac{1}{1-T(x_j)},
\end{equation}
and using Lagrange's Implicit Function Theorem (see, e.g., Theorem 1.2.4, page 17, in~\cite{gj0}), we obtain
\begin{equation}\label{Thphi}
\Th_{\{ j\} } \phi_{-1}(1) =  T(x_j), \qquad\qquad\qquad \Th_{\{ j\} } \phi_0(1) =  \frac{T(x_j)}{1-T(x_j)}.
\end{equation}

Now to prove that $\Hu_0(1,z,\bp)=G$. From~(\ref{hitzpde}) and~(\ref{ppG}), we obtain
\begin{align}
e^{\phi_0(s)} A(G)&= s \sum_{k\ge 1} \frac{1}{(k+1)!} \sum_{j=1}^k \binom{k+1}{j} \sum_{\ell\ge 0}\sum_{\substack{a_1,\ld ,a_j\ge 1\\a_1+\ld +a_j=\ell+1}} \!\!\!\!\!\! p_{a_1}\cd p_{a_j} [u^{\ell}] \big( T(us) - S(us)\big)^{k+1-j} \label{Heqn} \\
& \qquad\qquad\qquad \big. + p_1 s - e^{\phi_0(s)} \left(  \phi_{-1}(s) - \tfrac{1}{2} \phi_0(s)^2 \right). \notag
\end{align}
This is a formal power series in $\{p_is^i\}_{i\ge 1}$ with no constant term, and the sum of terms of total degree $n$, $n\ge 2$, is given by
\begin{align*}
& \sum_{k\ge n-1} \sum_{j=1}^k \frac{\binom{k+1}{j}}{(k+1)!} \sum_{\ell\ge 0}\sum_{\substack{a_1,\ld ,a_j\ge 1\\a_1+\ld +a_j=\ell+1}} \!\!\!\!\!\! p_{a_1}\cd p_{a_j} s^{\ell +1} [u^{\ell}] \binom{k+1-j}{n-j} T(u)^{k+1-n} \left(- S(u) \right)^{n-j} \\
   & \qquad - \frac{1}{(n-1)!}\phi_0(s)^{n-1} \phi_{-1}(s) + \frac{1}{2(n-2)!} \phi_0(s)^n \\
&= \sum_{j=1}^n \frac{1}{j!(n-j)!} \sum_{\ell\ge 0}\sum_{\substack{a_1,\ld ,a_j\ge 1\\a_1+\ld +a_j=\ell+1}} \!\!\!\!\!\! p_{a_1}\cd p_{a_j} s^{\ell +1} [u^{\ell}]  \left(- S(u) \right)^{n-j} \sum_{i\ge0} \frac{T(u)^i}{i!} \\
   & \qquad - \frac{1}{(n-1)!}\phi_0(s)^{n-1} \phi_{-1}(s) + \frac{1}{2(n-2)!} \phi_0(s)^n ,
\end{align*}
where for the last equality, we have changed from summation variable $k$ to $i=k+1-n$, and reordered the sums. Now we set $s=1$ and apply $\Th_{[n]}$ to this expression, via Lemma~\ref{umbcomp},~(\ref{Gsum}) and~(\ref{Thphi}). Using the notation $T_i$ to denote $T(x_i)$, $i=1,\ld ,n$, this gives
\begin{align*}
& \sum_{i=1} ^n x_i e^{T_i} \prod_{\substack{j=1\\j\ne i}}^n \left( \frac{x_j}{x_i-x_j} -  \Th_{\{ j\} } \left. S(x_i) \right|_{s=1} \right)  
- \sum_{i=1}^n T_i \prod_{\substack{j=1\\j\ne i}}^n \frac{T_j}{1-T_j} + \binom{n}{2} \prod_{j=1}^n\frac{T_j}{1-T_j}
\end{align*}
\begin{align*}
&= \sum_{i=1} ^n T_i  \prod_{\substack{j=1\\j\ne i}}^n  \frac{T_j(1+T_i-T_j)}{(1-T_j)(T_i-T_j)} 
- \left( \prod_{j=1}^n\frac{T_j}{1-T_j} \right) \left( \sum_{i=1}^n (1-T_i ) - \binom{n}{2} \right)\\
&= \left( \prod_{j=1}^n\frac{T_j}{1-T_j} \right) \left( \sum_{i=1}^n (1-T_i )\left( \prod_{\substack{j=1\\j\ne i}}^n \left( \frac{1}{T_i - T_j }+ 1 \right) - 1  \right) + \binom{n}{2}  \right) \\
&= \left( \prod_{j=1}^n \frac{T_j}{1-T_j} \right) \left( \sum_{\substack{\al \subseteq [n]\\ | \al | \ge 2}} \sum_{i \in \al} (1-T_i )\prod_{j \in \al\setminus i} \frac{1}{T_i - T_j } + \binom{n}{2}  \right) \\
&= \left( \prod_{j=1}^n \frac{T_j}{1-T_j} \right) \left( 0 - \binom{n}{2} +\binom{n}{2}  \right) = 0,
\end{align*}
where for the second last equality, we have applied Proposition~\ref{pfcnhm} in the variables $T_i$, for $i$ in $\al$. This completes our treatment of the terms of total degree $n\ge 2$ in~(\ref{Heqn}). Now, the sum of terms of total degree $1$ in~(\ref{Heqn}) is given by
\begin{equation*}
\sum_{k\ge 1} \frac{1}{(k+1)!} (k+1) \sum_{\ell\ge 0} p_{\ell + 1} s^{\ell +1} [u^{\ell}] T(u)^{k} +p_1 s - \phi_{-1}(s),
\end{equation*}
and we set $s=1$ and apply $\Th_{[1]}$ to this expression via~(\ref{Thphi}), to obtain $\;x_1e^{T_1} - T_1 =0$. From Proposition~\ref{identsym}, we see that this completes the proof that $A(G) = 0$, and since $\Hu_0(1,0,\bfp)=G(0,\bp)=0$, that the generating series for Hurwitz numbers in genus $0$ is given by $\Hu_0(1,z,\bfp)=G(z,\bp)$. 

\subsection{$m$-Hypermap numbers in genus $0$}\label{s92}

Define the series
\begin{equation}\label{Phim0explicit}
Q = Q(z,\bp) = \sum_{n\ge 1} z^n \sum_{\al\vdash n} \frac{p_{\al}}{| \Aut \al |} \, m \frac{ \left( (m-1)n-1  \right) ! }{\left( (m-1)n - l(\al) +2  \right) !} \prod_{j=1}^{l(\al)} \binom{m\al_j - 1}{\al_j} ,
\end{equation}
and recall from Section~\ref{s4} that $\Hy^{(m)}_0(1,z,\bp)$ is the generating series for $m$-hypermap numbers in genus $0$. Then Bousquet-M\'elou and Schaeffer~\cite{bms} proved bijectively that $\Hy^{(m)}_0(1,z,\bp)=Q(z,\bp)$. We now prove this result algebraically, by applying Theorem~\ref{genus0pde}. To do so, note that from Theorem~\ref{genus0pde}, and~(\ref{caseHy}), the generating series $\Hy^{(m)}_0(1,z,\bfp) = \whPsi^{(1+x)^m}_0$ satisfies the partial differential equation
\begin{equation}\label{mhyppde}
B(\Hy^{(m)}_0(1,z,\bp))=0
\end{equation}
where
\begin{align*}
B(F)&= z\sum_{k = 1}^m \binom{m}{k} \frac{1}{k+1} \sum_{j=1}^k \binom{k+1}{j} \sum_{\ell\ge 0}\sum_{\substack{a_1,\ld ,a_j\ge 1\\a_1+\ld +a_j=\ell+1}} \!\!\!\!\!\! p_{a_1}\cd p_{a_j} [u^{\ell}] \left(  \sum_{i\ge 1}(\pp_i F) u^i  \right)^{k+1-j} \!\!\!\!\!\!\!\!\!\!\!\! + p_1 z - \sum_{i\ge 1}p_i \pp_i F.
\end{align*}

Now let $w$ be given by the functional equation $w=z \, (1 + \psi(w))^{m-1}$, where $\psi(w) =\sum_{j\ge 1} \binom{mj - 1}{j} p_j w^j$. Slightly adapting the proof of Theorem~2.1 in~\cite{gs} we obtain
\[  \pp_i Q = \left( 1+\psi(w) \right) \frac{1}{i} \binom{mi}{i-1}w^i - \binom{mi}{i} w^i \sum_{j\ge 1} j \binom{mj-1}{j} p_j w^j \frac{1}{i+j}, \]
and thus we have
\begin{equation}\label{ptimesppQ}
\sum_{i\ge 1}p_i \pp_i Q = \left( 1+\psi(w) \right) \xi(w) - \frac{m}{2(m-1)} \psi(w)^2,
\end{equation}
where $\xi(w) = \sum_{i\ge 1} \tfrac{1}{i} \binom{mi}{i-1} p_i w^i$ and
\begin{equation}\label{ppQ}
\sum_{i\ge 1}(\pp_i Q) u^i = \left( 1+\psi(w) \right) R(uw) - Y(uw),
\end{equation}
where $R(x)=\sum_{i\ge 1}\tfrac{1}{i} \binom{mi}{i-1} x^i$ and $Y(x)= \sum_{i\ge 1} \binom{mi}{i} x^i \sum_{j\ge 1} j \binom{mj-1}{j} p_j w^j \frac{1}{i+j}$. Then $\tfrac{1}{i} \binom{mi}{i-1}$ is a generalization of the \emph{Catalan} number (the latter is the case $m=2$) and its generating series $R(x)$ satisfies the functional equation $R=x (1+R)^m$. Adapting the proof of Proposition~3.2(2) in~\cite{gj1}, we obtain
\begin{equation}\label{Qsum}
\Th_{\{ j\} } \left. Y(u) \right|_{w=1} =  \frac{u}{u-x_j} - \frac{R(u)}{R(u)-R(x_j)} \frac{1 + R(x_j)}{1- (m-1)R(x_j)},
\end{equation}
and using Lagrange's Implicit Function Theorem (see, e.g., Theorem 1.2.4, page 17, in~\cite{gj0}), we obtain
\begin{equation}\label{Thxi}
\Th_{\{ j\} } \psi(1) =  \frac{(m-1)R(x_j)}{1 - (m-1)R(x_j)}  ,\qquad\qquad \Th_{\{ j\} } \xi(1) =  R(x_j).
\end{equation}

Now to prove that $\Hy^{(m)}_0(1,z,\bp)=Q$. From~(\ref{mhyppde}),~(\ref{ptimesppQ}) and~(\ref{ppQ}), we obtain
\begin{align*}
(&1 + \psi(s))^{m-1} B(Q)\\
&\qquad = w \sum_{k = 1} ^m \binom{m}{k} \frac{1}{k+1} \sum_{j=1}^k \binom{k+1} {j} \sum_{\ell\ge 0}\sum_{\substack{a_1,\ld ,a_j\ge 1\\a_1+\ld +a_j=\ell+1}} \!\!\!\!\!\! p_{a_1}\cd p_{a_j} [u^{\ell}] \Big( \left( 1+\psi(w) \right) R(u w) \Big. \\
& \qquad \qquad \Big. - Y(uw) \Big)^{k+1-j} +p_1 w - \left( 1+\psi(w) \right)^{m}  \xi(w) - \frac{m}{2(m-1)} \left( 1+\psi(w) \right)^{m-1} \psi(w)^2 \\
&\qquad = m! \sum_{k=1}^m \frac{1}{(m-k)!} \sum_{j=1}^k \frac{1}{j!(k+1-j)!}  \sum_{\ell\ge 0}\sum_{\substack{a_1,\ld ,a_j\ge 1\\a_1+\ld +a_j=\ell+1}} \!\!\!\!\!\! p_{a_1}\cd p_{a_j} w^{\ell +1} [u^{\ell}] \Big( \left( 1+\psi(1) \right) R(u)\Big. \\
& \qquad\qquad \Big. - Y(u) \Big)^{k+1-j} +p_1w - \left( 1+\psi(w) \right)^{m}  \xi(w) - \frac{m}{2(m-1)} \left( 1+\psi(w) \right)^{m-1} \psi(w)^2.
\end{align*}
But this is a polynomial in $\{p_i w^i\}_{i\ge 1}$ with all terms of total degree at most $m+1$. Thus we set $w=1$ and apply $\Th_{[m+1]}$, to this expression, via Lemma~\ref{umbcomp},~(\ref{Qsum}) and~(\ref{Thxi}). Using the notation $R_i$ to denote $R(x_i)$, $i=1,\ld ,m+1$, this gives
\begin{align*}
& m! \left( \sum_{i=1} ^{m+1} x_i \prod_{\substack{j=1\\j\ne i}}^{m+1} \left( 1+ \frac{x_j}{x_i-x_j} + \left( 1 + \Th_{\{ j\} } \psi(1)\right) R_i -  \Th_{\{ j\} } \left. Y(x_i)\right|_{w=1} \right) \right. \\
& \qquad - \sum_{i=1}^{m+1} x_i \prod_{\substack{j=1\\j\ne i}}^{m+1}  \left(  1+ \frac{x_j}{x_i-x_j} \right) + h_1(x_1,\ld ,x_{m+1}) \\
& \qquad - \sum_{i=1}^{m+1}   R_i   \prod_{\substack{j=1\\j\ne i}}^{m+1}  \frac{1}{1 - (m-1)R_j}  
+ \left. (m-1) \!\! \left( \prod_{j=1}^{m+1}  \frac{1}{1 - (m-1)R_j} \right) \!\! \sum_{1\le i < n\le m+1} \!\!\!\!\!\! R_i R_n \right) \\
& = m! \left( \sum_{i=1} ^{m+1} x_i \prod_{\substack{j=1\\j\ne i}}^{m+1} \frac{R_i (1 + R_i)}{(1 - (m-1)R_j)(R_i - R_j)} \right. - \sum_{i=1}^{m+1} x_i^{m+1} \prod_{\substack{j=1\\j\ne i}}^{m+1} \frac{1}{x_i-x_j} + h_1(x_1, \ld ,x_{m+1}) \\
& \qquad -  \left. \left( \prod_{j=1}^{m+1}  \frac{1}{1 - (m-1)R_j} \right) \left( \sum_{i=1}^{m+1} R_i (1 - (m-1)R_i) -  (m-1) \!\!\!\!\!\! \sum_{1\le i < n\le m+1} \!\!\!\!\!\! R_i R_n \right)  \right)
\end{align*}
\begin{align*}
&= m! \left( \prod_{j=1}^{m+1} \frac{1}{1 - (m-1) R_j} \right) \Bigg( \sum_{i=1}^{m+1} ( R_i^{m+1} - (m-1) R_i^{m+2}) \prod_{\substack{j=1\\j\ne i}}^{m+1} \frac{1}{R_i - R_j } \Bigg. \\
& \qquad \qquad \qquad \qquad \qquad \qquad \qquad \Bigg. - h_1(R_1, \ld ,R_{m+1}) + (m-1) h_2(R_1, \ld ,R_{m+1}) \Bigg) = 0,
\end{align*}
where for the second last equality, we have applied Proposition~\ref{pfcnhm}, and for the last equality, we have applied Proposition~\ref{pfcnhm}, but in the variables $R_i$, for $i=1,\ld ,m+1$. From Proposition~\ref{identsym}, we see that this completes the proof that $B(Q) = 0$, and since $\Hy_0(1,0,\bfp)=Q(0,\bp)=0$, that the generating series for $m$-hypermap numbers in genus $0$ is given by $\Hy_0(1,z,\bfp)=Q(z,\bp)$. 

Recently, and independently, Fang~\cite{fang} proved that $\Hy_0(1,z,\bfp)$ satisfies this same partial differential equation $B(\Hy_0(1,z,\bfp))=0$, via a combinatorial analysis of the corresponding constellations. He gives an algebraic proof for $B(Q)=0$ that is quite different from our proof. In addition, Fang gives a partial differential equation for $m$-hypermap numbers in arbitrary genus.
 
\subsection{Monotone Hurwitz numbers in genus $0$}\label{s93}

{}From ~(\ref{casevecHu}), we can apply Theorem~\ref{Phif_gpde} with $f(x)=1$ and $g(x)=1-x$ for the generating series $\vec{\Hu} = \vec{\Hu}(y,z,\bp)= \Phi^{(1-x)^{-1}}$ for monotone Hurwitz numbers, and thus obtain the partial differential equation
\begin{equation}\label{monohurpde}
\cU_0 \, e^{\vec{\Hu}} = z^{-1} \left( \cC_0 - y \, \cC_1  \right) \, e^{\vec{\Hu}}.
\end{equation}
Now $\cU_0=p_1$, and from~(\ref{transcotrans}) we have $\cC_0 = z\frac{\partial}{\partial z}$, $\cC_1= 2\De$. Thus, adapting slightly the proof of Theorem~\ref{genus0pde}, we see that the generating series $\vec{\Hu}_0 = \vec{\Hu}_0(y,z,\bfp) $ for monotone Hurwitz numbers in genus $0$ satisfies the partial differential equation
\[ \frac{1}{2y} \left( z\frac{\partial}{\partial z} \vec{\Hu}_0 - zp_1 \right)  = \tfrac{1}{2} \sum_{i,j \ge 1} \left(  p_i p_j \pp_{i+j} \vec{\Hu}_0 + p_{i+j} \left( \pp_i \vec{\Hu}_0 \right) \left( \pp_j \vec{\Hu}_0 \right) \right). \]
But this equation for $\vec{\Hu}_0$ was given in~\cite{ggn1} (it follows immediately from Theorem 1.2 of that paper). The equation is called the {\em join-cut} equation, and it was obtained in~\cite{ggn1} by a combinatorial analysis of multiplication by a transposition factor. In~\cite{ggn1}, it was proved that
\begin{equation*}
\vec{\Hu}_0(1,z,\bp)  = \sum_{n\ge 1} z^n \sum_{\al\vdash n} \frac{p_{\al}}{| \Aut \al |} \, \frac{ \left( 2n+l(\al)-3 \right) ! }{\left( 2n \right) !} \prod_{j=1}^{l(\al)} \binom{2\al_j}{\al_j} ,
\end{equation*}
and we won't repeat such a proof here.

\section*{Acknowledgements}

We thank Guillaume Chapuy for pointing out reference~\cite{fang}, Valentin F\'eray for pointing out reference~\cite{ok}, and Mathieu Guay-Paquet for helpful discussions,.

\bibliographystyle{amsplain}

\end{document}